\documentclass[11pt,final]{amsart}
\usepackage{pst-all}
\usepackage{showkeys}
\usepackage[ps2pdf, pdfauthor={Benjamin Kennedy, Eugen Stumpf},%
pdftitle={Multiple
	slowly oscillating periodic solutions for $x'(t) = f(x(t-1))$ with negative feedback},%
colorlinks=true]{hyperref}

\theoremstyle{plain}
\newtheorem{thm}{Theorem}[section]
\newtheorem{prop}[thm]{Proposition}
\newtheorem{cor}[thm]{Corollary}
\newtheorem{lem}[thm]{Lemma}

\theoremstyle{remark}
\newtheorem{remark}[thm]{Remark}

\newcommand{\del}{\partial}
\newcommand{\N}{\mathbb{N}}

\newcommand{\R}{\mathbb{R}}

\title[Multiple slowly oscillating periodic solutions]{Multiple
	slowly oscillating periodic solutions for $x'(t) = f(x(t-1))$ with negative feedback}

\author[B. Kennedy]{Benjamin Kennedy}
\address{Gettysburg College\\Department of Mathematics \\
	300 North Wa\-shington St.\\Gettysburg, PA 17325-1400\\ United States}
\email{bkennedy@gettysburg.edu}

\author[E. Stumpf]{Eugen Stumpf}
\address{University Hamburg\\ Department of Mathematics\\
	Bundesstrasse 55\\20146 Hamburg\\Germany}
\email{eugen.stumpf@math.uni-hamburg.de}
\makeatletter
\@namedef{subjclassname@2010}{%
	\textup{2010} Mathematics Subject Classification}
\makeatother

\subjclass[2010]{Primary 34K13; Secondary 34C25}

\keywords{delay differential equations, periodic solutions, slowly oscillating, negative feedback}

{}
\copyrightinfo{2015}{B. Kennedy \& E. Stumpf}

\PII{}
\begin{document}

\begin{abstract}
We consider the prototype equation
\begin{equation*}
  x^{\prime}(t)=f(x(t-1))
\end{equation*}
for delayed negative feedback.  We review known results on uniqueness and nonuniqueness of slowly oscillating periodic solutions, and present some new results and examples.
\end{abstract}

\maketitle

\section{Introduction}
Consider the scalar-valued differential equation
\begin{equation}\label{eq: proto}
  x^{\prime}(t)=f(x(t-1))
\end{equation}
with constant delay $1$.  Throughout, we shall impose the following assumptions on the function $f$:
\begin{equation}\label{eq: assumptions}\tag{H}
  \left\{\begin{array}{l}\mbox{$f$ is continuous and bounded below or above}; \\[0.2cm]
\mbox{$f$ satisfies the negative feedback condition $xf(x) < 0$ for all $x \neq 0$}; \\[0.2cm]
\mbox{$f$ is $C^1$ on a neighborhood of $0$, with $f'(0) < 0$}.  \end{array} \right.
\end{equation}
Clearly, the continuity of $f$ implies $f(0)=0$; accordingly, the zero function $x(t)=0$, $t\in\R$, always satisfies Eq. \eqref{eq: proto} on all of $\R$.

Various applications have been proposed for Eq. \eqref{eq: proto}.
Nevertheless, it is undoubtedly the simplicity of Eq. \eqref{eq: proto} that has made it among the most intensively studied and best understood types of differential delay equations.

A \emph{slowly oscillating periodic solution} (SOP solution) of \eqref{eq: proto} is a periodic function $p: \R \to \R$ that satisfies Eq. \eqref{eq: proto} for all $t \in \R$ and has the feature that if $z \neq z'$ are distinct zeros of $p$, then $|z - z'| > 1$.  SOP solutions, when they exist, often play an important role in the global dynamics of Eq. \eqref{eq: proto}, and they have been studied for the last several decades.  Our goal in this paper is to review some known results on existence, uniqueness, and nonuniqueness of SOP solutions of Eq. \eqref{eq: proto}, to describe the prominent techniques used to obtain those results, and to provide some novel examples of equations with multiple SOP solutions.  In particular, among other examples we present a criterion for Eq. \eqref{eq: proto} to have multiple nontrivial SOP solutions when the equilibrium at $0$ is locally attracting, and a family of equations for which we can obtain multiple SOP solutions, one of which has arbitrarily long period, while keeping $f$ uniformly bounded.

In the usual way we write $C = C([-1,0],\R)$ for the Banach space of all continuous functions $\varphi:[-1,0]\to\R$ equipped with the norm $\|\varphi\|=\sup_{s\in[-1,0]}|\varphi(s)|$ of uniform convergence. If $x$ is any continuous real-valued function whose domain contains the interval $[t-1,t]$, the \emph{segment $x_t$ of $x$ at $t$} is the member of $C$ given by the formula $x_t(s) := x(t+s)$, $s \in [-1,0]$. A \emph{solution} of Eq. \eqref{eq: proto} is either a continuously differentiable function $x: \R \to \R$ that satisfies \eqref{eq: proto} for all $t\in\R$, or a continuous function $x : [-1,\infty)\to\R$ that is continuously differentiable for all $t>0$ and that satisfies Eq. \eqref{eq: proto} for all $t > 0$.  In either case, $x_{0}$ is called the \emph{initial value} of the solution $x$.  We regard $C$ as the phase space for Eq. \eqref{eq: proto}, and view solutions $x$ as describing orbits $\{x_t\}$ in $C$.

By the so-called \emph{method-of-steps}, it is easily seen that every function $\varphi\in C$ determines a solution $x^{\varphi}:[-1,\infty)\to\R$ of Eq. \eqref{eq: proto} with initial value $x_{0}^{\varphi}=\varphi$.  Indeed, for $t \in [0,1]$, we have the formula
\[
x^\varphi(t) =  \varphi(0) + \int_0^t f(\varphi(s - 1)) \ ds,
\]
and the formulas for $x^\varphi$ on $[1,2], \ [2,3], \ \ldots$ are then similar.  We call $x^\varphi$ the \emph{continuation of $\varphi$ as a solution of Eq. \eqref{eq: proto}}.  Eq. \eqref{eq: proto} defines a semiflow
\begin{equation*}
  F:[0,\infty)\times C\ni(t,\varphi)\longmapsto x^{\varphi}_{t}\in C
\end{equation*}
on the Banach space $C$. $F$ is continuous in the following sense: given $\varphi\in C$, $t\geq 0$, and $\varepsilon>0$, there exists $\delta>0$ such that for all $\psi\in C$ with $\|\psi-\varphi\|< \delta$ and all $0\leq s\leq t$,
\[
|F(s,\psi) - F(s,\varphi)| = |x^{\psi}(s)-x^{\varphi}(s)|<\varepsilon.
\]
Additionally, using the Arzela-Ascoli theorem it not difficult to show that all maps $F(t,\cdot): C \to C$, $t\geq 1$, are compact in the sense that they map bounded sets to precompact sets.

We define a solution $x$ to be \emph{slowly oscillating} if for any two zeros $z < z'$ of $x$ we have $z' - z > 1$. Similarly, a solution $x$ is called \emph{eventually slowly oscillating} if there is some $\tau\in\R$ such that for any two zeros $\tau<z<z'$ of $x$ it holds that $z'-z>1$.  Suppose that $\varphi \in C$ is such that (for example) $\varphi(s) > 0$ for all $s \in [-1,0]$.  Then either $x^\varphi(t) > 0$ for all $t > 0$ or not; in the latter case, there is some first positive zero $z$ of $x^\varphi$.  The negative feedback condition then implies that $f$ is strictly decreasing on $[z,z+1]$. Thus the next zero $z'$ after $z$, if it exists at all, satisfies $z' - z > 1$; repeating this argument shows that any successive zeros of $x^\varphi$ must be separated by more than one unit.  Thus eventually slowly oscillating solutions clearly exist.  Indeed, much more is true: if $f$ is smooth and strictly decreasing, the solution $x^\varphi$ is eventually slowly oscillating for a dense subset of initial values $\varphi \in C$ (see Mallet-Paret and Walther \cite{MPW}, and the earlier but somewhat more restricted work of Walther \cite{Walther 1981}); a similar conclusion seems plausible for Eq. \eqref{eq: proto}.  (We point out, however, that examples are known of equations of the form
\[
x'(t) = -\mu x(t) + f(x(t - 1)), \ \mu \geq 0
 \]
--- that is, like Eq. \eqref{eq: proto} with an instantaneous damping term added --- for which open sets of initial conditions have continuations that are not eventually slowly oscillating.  See, for example, Ivanov and Losson \cite{Ivanov and Losson 1999} and Stoffer \cite{Stoffer 2008}.  For such examples, $f$ must be non-monotone: again see \cite{MPW}.)

The zero function $0 \in C$ is a stationary point of the semiflow $F$: $F(t,0)=0$ for all $t\geq 0$. The linear problem associated with Eq. \eqref{eq: proto} at $0$ is the equation
\begin{equation}\label{eq: linearization}
y'(t) = f'(0)\,y(t-1).
\end{equation}
Using the ansatz $y(t)=e^{\lambda t}$ we obtain the corresponding characteristic equation
\[
\lambda - f'(0)\, e^{-\lambda} = 0.
\]
The roots of this characteristic equation coincide with the eigenvalues of the linearization of $F$ at $0 \in C$. It is also known that, when $f^{\prime}(0)<0$, there are only finitely many eigenvalues with nonnegative real part; there are eigenvalues with positive real part if and only if $f'(0)<-\pi/2$.  Accordingly, the zero solution $x(t)=0$, $t\in\R$, of Eq. \eqref{eq: proto} is locally asymptotically stable when $0>f'(0) > -\pi/2$ and unstable when $f'(0) < -\pi/2$ by the so-called principle of linearized stability and instability, respectively.  In the former case, the trivial solution of Eq. \eqref{eq: linearization} is globally attracting and the trivial stationary point of Eq. \eqref{eq: proto} is at least locally attracting. As we are interested in slowly oscillating periodic solutions we of course will only consider examples in which the trivial solution of Eq. \eqref{eq: proto} is at most locally, but not globally, attracting.  (Several authors have given conditions under which the trivial solution \emph{is} globally attracting; for example, in \cite{LR} Liz and R\"ost prove if $f$ is $C^3$, strictly decreasing, has negative Schwarzian derivative everywhere, and $f'(0) \geq -3/2$, then the zero solution $x(t)=0$, $t\in\R$, is globally attracting.)

\section{SOP solutions --- essential methods and results}

For general facts about delay differential equations we recommend the monograph \cite{Diekmann} of Diekmann et al..

\subsection*{The map $P: K \to K$}

For the remainder of this paper we restrict our attention to the case that $f$ is bounded above; the case where $f$ is bounded only below is essentially the same.  In particular, let us take $M > 0$ such that $f(x) \leq M$ for all $x \in \R$.

Let us write
\[
 K:= \{ \ \varphi \in C \ | \ \varphi(-1) = 0, \ \varphi \ \mbox{is nondecreasing}, \ \|\varphi\| \leq M \ \}.
\]
$K$ is a closed convex subset of $C$.

The following theorem is elementary and mostly intuitively clear.  For a proof we refer the reader to Diekmann et. al. \cite[Chapter XV]{Diekmann}, or to \cite{Walther 1995} where the more general equation $x^{\prime}(t)=-\mu\,x(t)+f(x(t-1))$ with $\mu\geq 0$ is discussed.

\begin{prop}\label{prop: basic 2}
  Let $\varphi\in K$, $\varphi \neq 0$.  Then $x=x^{\varphi}$ is slowly oscillating on $[0,\infty)$. If $z>0$ is a zero of $x$ then $|x|$ and $|x^{\prime}|$ are bounded on the interval $[z,z+1]$ by
        \begin{equation*}
          \max\lbrace|f(\xi)|\mid\xi\in x([z-1,z])\rbrace,
        \end{equation*}
        and the function $s \longmapsto |x_{z+1}(s)|$ is increasing on $[-1,0]$.

        If the zeroset of $x\vert_{[0,\infty)}$ is unbounded then it is given by a sequence of points $z_{j}=z_{j}(\varphi)$, $j\in\N$, with
        \begin{equation}\label{eq: so}\tag{so}
          z_{j}+1<z_{j+1}\qquad\text{and}\qquad x^{\prime}(z_{j})\not=0\text{ for all }j;
        \end{equation}
        and $x$ is monotone on $[0,z_{1}+1]$ and on each interval $[z_{j}+1,z_{j+1}+1]$.

        If the zeroset of $x\vert[0,\infty)$ is bounded then it is given by a finite sequence of points $z_{j}=z_{j}(\varphi)$, $1\leq j\leq J=J(\varphi)$, with property \eqref{eq: so}, and $|x(t)|$ decreases monotonically to $0$ on the interval $[z_{J}+1,\infty)$ as $t\to\infty$.
\end{prop}

Given initial value $\varphi \in K$, $\varphi \neq 0$, there are two possibilities: either $x^\varphi$ has a second positive zero $z$ and $x^\varphi_{z+1} \in K$, or not.  We define the return map $P: K \to K$ by setting $P(0) = 0$, $P(\varphi): = x^\varphi_{z+1}$, if $z>0$ exists, and $P(\varphi) := 0$ otherwise.  Here are the basic facts about $x^\varphi$ and the map $P$ (for a proof, see Lemma 2.7 in Nussbaum \cite{RDN 1973}).

\begin{prop}\label{prop: return map}
The map $P$ is continuous, and $P(K)$ has compact closure.
\end{prop}

Suppose that $\varphi$ is a nonzero periodic point of $P$ with minimal period $k$.  Then, writing $0 < z_1 < \ldots < z_{2k} < \cdots$ for the positive zeros of $x^\varphi$, we have that $x^\varphi_{z_{2k}+1} = \varphi$ and that $x$ is a periodic solution of Eq. \eqref{eq: proto} with minimal period $z_{2k}+1$.  Conversely, any slowly oscillating periodic solution of Eq. \eqref{eq: proto} has a segment that is a periodic point of $P$.  Accordingly, perhaps the most prominent technique for proving the existence of SOP solutions of nonlinear equations Eq. \eqref{eq: proto} and its generalizations has been to show that analogs of the return map $P : K \to K$ have nonzero fixed points.  Since $0 \in K$ is a fixed point of $P$, merely establishing the existence of at least one fixed point (by a naive application of Schauder's Theorem, for example) is insufficient.

We do, however, have the following theorem.  This is perhaps the best-known theorem on existence of SOP solutions of \eqref{eq: proto} and was given by Nussbaum in 1974 \cite{RDN 1974}.

\begin{thm}\label{thm:first}
If, in addition to the hypotheses \eqref{eq: assumptions}, we assume that $f'(0) < -\pi/2$ (and so the equilibrium at the origin is unstable), then $P$ has a nonzero fixed point.
\end{thm}

The idea of Nussbaum's proof is to show that the instability of the origin makes $0 \in K$ a so-called \emph{ejective fixed point} of $P$ and then to apply the so-called Browder ejective fixed point principle \cite{Browder 1965}, which implies that $P$ must have a fixed point that is not ejective.  This same basic idea has since been used to prove the existence, when the equilibrium is unstable, of nontrivial periodic solutions in several generalizations and extensions of Eq. \eqref{eq: proto}: these generalizations include versions of \eqref{eq: proto} with instantaneous damping (Hadeler and Tomiuk \cite{Hadeler and Tomiuk 1977}) and various equations with state-dependent delay (e.g. Nussbaum \cite{RDN 1974}, Alt \cite{Alt 1979}, Kuang and Smith \cite{Kuang and Smith 1992}, and Walther \cite{Walther 2008}).

We now turn to some other important methods for studying SOP solutions of Eq. \eqref{eq: proto}.

\subsection*{When $f$ is odd --- phase plane methods}

An entirely different approach to proving existence of SOP solutions in the case that $f$ in Eq. \eqref{eq: proto} is odd was pioneered by Kaplan and Yorke in 1974 \cite{Kaplan and Yorke 1974}.  In this case the coupled two-dimensional system of ordinary differential equations
\begin{equation}\label{eq: ODE}
  \left\lbrace\begin{aligned}
    u^{\prime}&=f(v)\\
    v^{\prime}&=-f(u)
  \end{aligned}\right.
\end{equation}
has a first integral of the form
\begin{equation}\label{eq: Hamiltonian}
  H(u,v):=-\int_{0}^{u}f(s)\,ds-\int_{0}^{v}f(s)\,ds.
\end{equation}

Eq. \eqref{eq: ODE} is highly symmetric: if $(u(\cdot),v(\cdot)):I\to\R^{2}$, $I\subset\R$ an interval, is a solution of Eq. \eqref{eq: ODE}, then $(-u(\cdot),-v(\cdot)):I\to\R^{2}$, $(v(\cdot),-u(\cdot)):I\to\R^{2}$, and $(-v(\cdot),u(\cdot)):I\to\R^{2}$ are also solutions of Eq. \eqref{eq: ODE}, as simple calculations show. Note that $H(0,0)=0$ whereas $H(u,v)>0$ for all $(u,v)\not=(0,0)$, and $\nabla H(u,v)=0\in\R^{2}$ if and only if $(u,v)=(0,0)$; it follows that, given any real $\alpha>0$ with $\alpha\in H(\R^{2}):=\lbrace H(u,v)\in[0,\infty)\mid u,v\in\R\rbrace$, the level set $H^{-1}(\lbrace \alpha\rbrace )$ is a simple closed curve around the origin in $\R^{2}$. This curve is the orbit of a periodic and nonconstant solution $(u(\cdot),v(\cdot)):\R\to\R^{2}$ of Eq. \eqref{eq: ODE}.  From the symmetry of $H$ we conclude that in this situation the orbits of $(u(\cdot),v(\cdot))$, $(-u(\cdot),-v(\cdot))$, $(v(\cdot),-u(\cdot))$, and $(-v(\cdot),u(\cdot))$ coincide; that is, the orbit $H^{-1}(\lbrace \alpha\rbrace)$ is invariant under rotations by $\pi/2$, and all the solutions $(-u(\cdot),-v(\cdot))$, $(v(\cdot),-u(\cdot))$, and $(-v(\cdot),u(\cdot))$ are translations of the periodic solution $(u(\cdot),v(\cdot))$.

 It was observed by Kaplan and Yorke in \cite{Kaplan and Yorke 1974} that any solution $(u(\cdot),v(\cdot)):\R\to\R^{2}$ of Eq. \eqref{eq: ODE} with period four is associated to a periodic solution of Eq. \eqref{eq: proto}: namely, $x(t):=u(t)$, $t\in\R$, satisfies Eq. \eqref{eq: proto}, is periodic of minimal period four, and has the special symmetry $x(t)=-x(t-2)$, $t\in\R$. We shall refer to such a solution as a \emph{Kaplan-Yorke solution} of Eq. \eqref{eq: proto}.

The existence result obtained in \cite{Kaplan and Yorke 1974} for Eq. \eqref{eq: proto} is the following.  The hypotheses guarantee that Eq. \eqref{eq: ODE} has a solution of period four.

\begin{thm}\label{KYTHM}
Assume that, in addition to hypothesis \eqref{eq: assumptions}, $f$ is odd and that $\int_0^\infty |f(x)| \ dx = \infty$.  Write
\[
a = \lim_{x \to 0} \frac{f(x)}{x} \ \ \mbox{and} \ \ A = \lim_{x \to \infty} \frac{f(x)}{x}.
\]
If either $A < -\pi/2 < a$ or $a < -\pi/2 < A$, then \eqref{eq: proto} has a Kaplan-Yorke solution.
\end{thm}

In \cite{RDN 1979}, Nussbaum combines the above theorem with Schauder's theorem --- roughly speaking applied on a subset of a variant  of $K$ --- to exhibit an example where Kapan-Yorke solutions coexist with solutions of period greater than 4.  We will present some other examples of a similar flavor below.

\subsection*{When $f$ is strictly decreasing}

As already suggested in the last section, the case that $f$ is smooth and strictly decreasing is particularly well understood.  A first major step in this understanding was made by Kaplan and Yorke in \cite{Kaplan and Yorke 1975}, who considered the orbits traced out by slowly oscillating solutions in the $(x(t),x(t-1))$ plane.  The chief observation was that the ways in which these orbits can cross in the plane is sharply limited.  The main theorem in \cite{Kaplan and Yorke 1975}, says, roughly speaking, that if $f'(0) < -\pi/2$ these orbits tend toward an annulus in the plane whose inner and outer boundary curves are the orbits of SOP solutions.  In particular, if Eq. \eqref{eq: proto} has a unique SOP solution in this case, this SOP solution is necessarily asymptotically stable.

In \cite{RDN 1979}, Nussbaum built on the phase plane approach in \cite{Kaplan and Yorke 1975} to give conditions under which this SOP solution is unique (it is sufficient, for example, that $f$ be smooth and odd with $f'(x)$ and $f(x)/x$ increasing on the positive half line).  Also using a phase plane approach, Cao \cite{Cao 1996} established uniqueness of SOP solutions for $f$ smooth and strictly decreasing and with a slighty different concavity condition, but without the assumption that $f$ is odd.

The behavior of all slowly oscillating soltions of Eq. \eqref{eq: proto} in the case that $f$ is smooth and strictly monotonic has come to be thoroughly understood; see, for example, Walther \cite{Walther 1995}.  Among other results, for example, it is known that in this case all SOP solutions correspond to fixed points (as opposed to higher-period periodic points) of the map $P: K \to K$.

\subsection*{The fixed point index}

Other existence results, as well as uniqueness and nonuniqueness results, have depended on the so-called fixed point index.  We shall use the index below and give some of its flavor; very roughly speaking, the index sometimes allows us to use information on subsets on $K$ where we understand $P$ well to draw conclusions about the existience of fixed points on subsets of $K$ where we understand $P$ less well.  We mention here a handful of earlier results.  In \cite{RDN 1973}, Nussbaum used a fixed point index approach to establish the existence of SOP solutions under hypotheses similar to those in Theorem \ref{KYTHM}, but without the assumption that $F$ is odd; our work in Section \ref{SEC3} below bears a strong conceptual affinity with Nussbaum's approach.  In \cite{Xie 1991}, Xie combined stability estimates and the fixed point index to obtain, in certain cases, uniqueness of fixed points of $P$.

\subsection*{Stability of SOP solutions}

Assessing the stability of SOP solutions is, in general, a difficult problem.  The stability implied by uniqueness in the case that $f$ is monotonic, alluded to above, is one result; and various authors (see, for example, Cao \cite{Cao 1995}) have used Kaplan-Yorke-type phase plane techniques with $f$ monotonic to conclude, based on how solutions ``spiral'' in the $x(t),x(t-1)$ plane, that certain SOP solutions are stable or unstable.  Other authors have observed that, loosely speaking, if $f$ is sufficiently flat on long enough intervals, the return map $P$ becomes contractive (at least on certain subsets of $K$) and so we can conclude that asymptotically stable SOP solutions exist.  This basic idea has been used by several authors; see for example, Walther \cite{Walther 2001} and Vas \cite{Vas 2011}.

Several authors have also established results on the Floquet multipliers of SOP solutions, especially in the case that $f$ is odd.  (If $p$ is an SOP solution with $p_0 \in K$, the Floquet multipliers of $p$ essentially coincide with the spectrum of the derivative of $P$ at $p_0$.)   See, for example, Skubachevskii and Walther \cite{Skubachevskii and Walther 2006} and the references therein.

\subsection*{Other results on nonuniqueness}

Peters \cite{Peters 1983} and Siegberg \cite{Siegberg 1984} have exhibited examples of Eq. \eqref{eq: proto} where $F$ is semiconjugate, on a suitable subset of $K$, to a chaotic interval map.  For such equations, of course, $P$ has many periodic points of many different periods.

Cao \cite{Cao 1995} and Vas \cite{Vas 2011} have both given examples of equations with many SOP solutions, many of which are stable.  In \cite{Cao 1995}, $f$ is monotonic, and the above-mentioned results of Kaplan and Yorke on the behavior of orbits in the $(x(t),x(t-1))$ plane are exploited to devise a condition that guarantees the existence of several (possibly infinitely many) SOP solutions whose planar orbits are nested within one another.  In \cite{Vas 2011} $f$ is unbounded and similar to a decreasing step function (the equation in \cite{Vas 2011} also has a damping term added).  The many intervals where $f$ is nearly flat causes the analog of $P$ to be contractive on many different subdomains of $C$, resulting in many distinct periodic solutions.  In Section \ref{SEC5} we present an example in a similar spirit, but the feedback function is bounded and the non-monotonicity of $f$ plays a key role.

\section{A class of equations with at least two SOP solutions}\label{SEC3}

\subsection*{Facts about the fixed point index}  We here recall only those properties of the fixed point index that we shall need and even those in a restricted setting. For a deeper discussion consult Granas and Dugundji \cite[Chapter IV]{Granas 2003} and also Nussbaum \cite[Section 1]{RDN 1974}.

Suppose that $X$ is a closed, bounded, convex subset of a Banach space.  $X$ has the subspace topology induced by the Banach space topology.  (This is the topology to which we shall be referring throughout this paragraph --- so the open sets we refer to just below are open relative to $X$, and may have empty interior with respect to the underlying Banach space.)  Suppose that $G: X \to X$ is continuous and compact in the sense that it maps bounded sets of $X$ into precompact sets.  Then, if $V$ is any open (relative to $X$) subset of $X$ such that $G$ is fixed-point free on $\del V := \overline{V} \setminus V$, the integer $i_X(G,V)$, \emph{the fixed point index of $G$ on $V$ with respect to $X$}, is defined.  The following two properties hold:
\emph{
\begin{itemize}
\item[(I)] If $V\subset X$ is convex and $G(\overline{V}) \subset V$, then $G$ has a fixed point in $V$ and $i_X(G,V) = 1$.  In particular, $i_X(G,X) = 1$.
\item[(II)] If $i_X(G,U)$ and $i_X(G,V)$ are defined, $\overline{U} \cap \overline{V} = \emptyset$, and the fixed point set of $G$ is contained in $U \cup V$, then
    \begin{equation*}\label{property II}
    1 = i_X(G,X) = i_X(G,U) + i_X(G,V).
    \end{equation*}
\end{itemize}}

Property (I) follows from the so-called normalization and homotopy properties of the index, and property (II) is the so-called additivity property of the index. We will also need the so-called mod $p$ theorem for the fixed point index.  The following is a very special case of this theorem where a general statement is contained, for instance, in Granas and Dugundji \cite[p. 460]{Granas 2003}:

\begin{lem}\label{lem: mod p theorem}
Suppose that $X$ is a closed, bounded, convex subset of a Banach space and that $G: X \to X$ is a continuous compact map.  Let $p \in \N$ be prime and write $Fix(G)$ and $Fix(G^p)$ for the fixed point sets of $G$ and $G^p$ in $X$, respectively.  Suppose that $V \subset X$ is relatively open and that $Fix(G) \cap \overline{V} = Fix(G^p) \cap \overline{V} = \{q\}$, $q \in V$.

Then $i_X(G,V)$ is congruent to $i_X(G^p,V)$ modulo $p$.
\end{lem}

As a consequence of this result we obtain the next corollary, which will be essential for our approach.

\begin{cor}\label{cor: ind}
Suppose that $X$ is a closed, bounded, convex subset of a Banach space and that $G: X \to X$ is a continuous compact map.  Let $U$ be a relatively open convex subset of $X$.  Suppose that $q \in U$ is a fixed point of $G$ and that there is some $n \in \N$ such that, for all integers $m \geq n$,
\begin{equation*}
G^m(\overline{U}) \subset U\qquad\text{ and }\qquad Fix(G^m) \cap \overline{U} = \{q\}.
\end{equation*}
  Then $i_X(G,U) = 1$.
\end{cor}

\begin{proof}
Since $Fix(G) \cap \overline{U} \subset Fix(G^m) \cap \overline{U}$ for all $m \geq n$ and $q \in Fix(G)$ by assumption, we actually have that $Fix(G) \cap \overline{U} = Fix(G^m) \cap \overline{U} = \{q\}$ for all such $m$.  By property (I) of the fixed point index given above, $i_X(G^m,U) = 1$ for all sufficiently large $m$, and in particular for all sufficiently large primes $p$.  Thus by Lemma \ref{lem: mod p theorem} $i := i_X(G,U)$ is congruent to $1$ mod $p$ for all sufficiently large primes $p$.  It follows that $i$ must equal $1$:  for let $p$ be any prime greater than $m$.  To say that $i \equiv 1$ mod $p$ is to say that $i = kp + 1$ for some integer $k$.  Now choose a prime $p' > |kp+1|+1$; if $k \neq 0$, then $kp+1$ cannot be congruent to $1$ mod $p'$.
\end{proof}

We are now in a position to discuss our first example of an instance of Eq. \eqref{eq: proto} with at least two distinct slowly oscillating periodic solutions. An important aspect of this example is the fact that the trivial solution $x(t)=0$, $t\in\R$, is locally asymptotically stable.  The result thus bears some intuitive resemblence to the $f'(0) > -\pi/2$ case of Theorem \ref{KYTHM} (and even more so the closely related Theorem 2.1 in \cite{RDN 1973}) in that there is a periodic solution that is in some sense ``between'' two stable sets.   Here are the main ideas, which can be applied under several sets of hypotheses other than those we formulate below.  We will apply the fixed point index to the continuous, compact map $P$ on the closed, bounded, convex set $K$.  We will exhibit two relatively open convex subsets $V$ and $U$ of $K$ such that $\overline{V} \cap \overline{U} = \emptyset$, $P(\overline{V}) \subset V$, and $U$ satisfies the hypotheses of Corollary \ref{cor: ind} ($U$ will be an open set containing $0$).  We will accordingly have $i_K(P,U) = i_K(P,V) = 1$; the contrapositive of (II) will then guarantee the existence of another nontrivial fixed point of $P$ in $K \setminus (U \cup V )$.  Since $P$ is fixed-point free on $\del V$ and $\del U$, this fixed point actually lies in the open set $K \setminus (\overline{U} \cup \overline{V} )$.

For the rest of this section we impose, in addition to \eqref{eq: assumptions}, the following additional restriction on the feedback function $f:\R\to\R$:
\begin{equation}\label{eq: assumptions 2}\tag{H'}
  \left\{\begin{array}{l}\mbox{$f$ is bounded by some $\mu>0$}; \\[0.2cm]
\mbox{There are constants $\beta,\sigma>0$ such that $|f(x)|\geq\sigma$ whenever $|x|\geq \beta$.}  \end{array} \right.
\end{equation}

Here is the main result of the section.

\begin{prop}\label{prop: example one}
Suppose that (H) and (H') hold, and moreover that
\begin{equation*}
-\frac{\pi}{2}<f'(0)<0\qquad\text{ and }\qquad
\beta < \frac{\sigma}{2 + \mu/\sigma}.
\end{equation*}
Then Eq. \eqref{eq: proto} has at least two distinct slowly oscillating periodic solutions.  Each of these solutions corresponds to a fixed point of $P$ (and so has minimal period given by the length of a minimal interval containing three consecutive zeros).
\end{prop}

The first step in the proof of Proposition \ref{prop: example one} is to exhibit the subset $V \subset K$ described above.

\begin{lem}\label{prop: V}
Suppose that $(H)$ and $(H')$ hold, and that
\[
\beta < \frac{\sigma}{2 + \mu/\sigma}.
\]
Then there is a relatively open convex subset $V \subset K$ such $0 \notin \overline{V}$ and such that $P(\overline{V}) \subset V$.  Accordingly, $i_K(P,V) = 1$ and $P$ has a fixed point in $V$; this fixed point is a segment of a slowly oscillating periodic solution of Eq. \eqref{eq: proto}.
\end{lem}

\begin{proof}
Note that our conditions on $\beta$ guarantee that $\beta/\sigma < 1/2$.  Choose $\gamma > 0$ small enough that $\gamma + \beta/\sigma < 1$ and
\begin{equation}\label{eq: aux-ineq}
\beta \leq \frac{\sigma - (\sigma + \mu)\gamma}{2 + \mu/\sigma}.
\end{equation}
Now define the following set:
\[
V := \left\lbrace\varphi \in K \mid \ \varphi(t) > \sigma t + \sigma(1-\gamma) \ \mbox{for} \ t \in [-1+\gamma,-\beta/\sigma] \ \right\rbrace.
\]
This set is clearly open and convex, and its closure is
\[
\overline{V} = \left\lbrace\varphi \in K \mid \ \varphi(t) \geq \sigma t + \sigma(1-\gamma) \ \mbox{for} \ t\in [-1+\gamma,-\beta/\sigma] \ \right\rbrace.
\]
The members of $\overline{V}$ are all the members of $K$ whose graphs do not go below the graph pictured in Fig. \ref{fig: definition of V}.

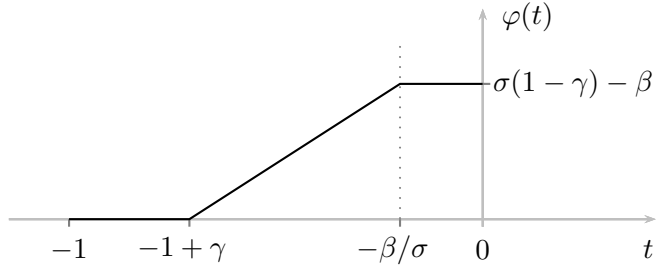
\begin{figure}[htb]
\centering
\psset{unit=1.cm}
\begin{pspicture}(-6.5,-1)(2.5,3.4)
\psaxes[labels=none,linecolor=lightgray,ticks=none]{->}(0,0)(-6.3,-0.1)(2.3,2.8)
\psline[linecolor=gray](-5.5,0)(-5.5,-0.1)
\psline[linecolor=gray](-3.9,0)(-3.9,-0.1)
\psline[linestyle=solid](-5.5,0)(-3.9,0)(-1.1,1.8)(0,1.8)
\psline[linecolor=gray](-1.1,0)(-1.1,-0.1)
\psline[linecolor=gray](0,1.8)(0.1,1.8)
\psline[linecolor=gray, linestyle=dotted](-1.1,0)(-1.1,2.3)
\rput(1.2,1.8){$\sigma(1-\gamma)-\beta$}
\rput(-1.2,-0.4){$-\beta/\sigma$}
\rput(-5.5,-0.4){$-1$}
\rput(-4,-0.4){$-1+\gamma$}
\rput(0,-0.4){$0$}
\rput(0.6,2.7){$\varphi(t)$}
\rput(2.2,-0.4){$t$}
\end{pspicture}
\caption{Definition of $\overline{V}$}\label{fig: definition of V}
\end{figure}

Suppose now that $x:[-1,\infty)\to\R$ is a solution of Eq. \eqref{eq: proto} with $x_0 \in \overline{V}$.  Notice that $x(0) \geq \sigma(1 - \gamma) - \beta$, and that this latter quantity is greater than $\beta$ by assumption \eqref{eq: aux-ineq}. Moreover, $x(t) \geq \beta$ for all $t \in [-1+(\gamma + \beta/\sigma),0]$ since $x_{0}\in\overline{V}$.  We claim that in fact $x(t) \geq \beta$ for all $t \in [0,\gamma+\beta/\sigma]$ as well.  For by the first point of $(H')$, we have $x'(t) \geq -\mu$ for all $t \in [0,\gamma + \beta/\sigma]$, and so accordingly we have
\[
x(\gamma + \beta/\sigma) \geq \sigma(1 - \gamma) - \beta - \mu(\gamma + \beta/\sigma);
\]
our assumption \eqref{eq: aux-ineq} on $\beta$ guarantees that the quantity on the right is greater than or equal to $\beta$.  Therefore, at the first positive time $\tau$ such that $x(\tau) = \beta$, we see that $x(t) \geq \beta$ on the entire interval $[\tau- 1,\tau]$.

Hypothesis \eqref{eq: assumptions 2} and the negative feedback condition now yield that $x'(t) \leq -\sigma$ for all $t \in [\tau,\tau+1]$.  Thus $z_1$, the first positive zero of $x$, occurs in the interval $(\tau,\tau+\beta/\sigma]$.  We therefore see that $x'(t) \leq -\sigma$ on the interval
\[
[z_1,\tau+1] \supset [z_1,z_1+1-\beta/\sigma]
\]
and that $x'(t) \leq 0$ on $[\tau+1,z_1+1]$.  It follows that
\[
-x_{z_1+1}(s) \geq \sigma s+\sigma > \sigma s + \sigma(1-\gamma) \ \mbox{for all} \ s \in [-1,-\beta/\sigma]
\]
and so, in particular, $x_{z_1 + 1} \in -V$.

By taking into account that the assumptions \eqref{eq: assumptions 2} are symmetric, even though $f$ is not assumed odd, a symmetric argument shows that $x_{z_2 + 1} \in V$, which is the desired conclusion.
\end{proof}

\begin{remark}
	Let $p:\R\to\R$ denote the slowly oscillating periodic solution found in the last result, translated so that $p_0 \in V \subset K$. By imposing further conditions on $f(x)$ for $|x| \geq \beta$, one can guarantee that $p_0 \in V$ is a stable fixed point of $P$. One particularly simple such condition is to insist that $f$ be constant on $[\beta,\infty)$.  In this case, the proof of Lemma \ref{prop: V} shows that $P$ is in fact constant on $\overline{V}$ since, in the notation of the above proof, $x_\tau$ can only be continued in one possible way.  More generally, by insisting that $f$ be Lipschitz and have a small enough Lipschitz constant on $[\beta,\infty)$, it is possible to show, in the spirit of Walther \cite{Walther 2001}, that $P$ is contractive on the subset $\overline{V}$ of $K$.
\end{remark}

We now turn our attention to our desired open convex subset $U \subset K$ about $0\in K$.  The following lemma follows from standard results about the behavior of solutions of Eq. \eqref{eq: proto} near the trivial solution $x(t)=0$, $t\in\R$, when the latter is locally asymptotically stable by the so-called principle of linearized stability.

\begin{lem}\label{prop: U}
Suppose that $(H)$ holds and that $f'(0) \in (-\pi/2,0)$.  Then there is a an open convex subset $U$ of $K$ such that $0 \in U$ and such that the hypotheses of Corollary \ref{cor: ind} are satisfied. More precisely, there is an open convex set $U\subset K$ with $0\in U$ and some $n \in \N$ such that, for all $m \geq n$,
\begin{equation*}
P^m(\overline{U}) \subset U\qquad\text{ and }\qquad Fix(P) \cap \overline{U} = Fix(P^m) \cap \overline{U} = 0.
\end{equation*}
Thus $i_{K}(P,U) = 1$.

Moreover, $U\subset K$ can be chosen so that $\overline{U} \cap \overline{V} = \emptyset$ where $V$ is as in Proposition \ref{prop: V}.
\end{lem}

\begin{proof}
Since $f'(0) \in (-\pi/2,0)$, $0\in C$ is a locally exponentially stable stationary point of the semiflow $F$.  Thus there are numbers $\epsilon > 0$, $\kappa > 0$, and $\omega < 0$ such that $\|F(t,\varphi)\| \leq \kappa e^{\omega t}$ whenever $\|\varphi\| \leq \epsilon$ (see, for instance, \cite[Chapter VII]{Diekmann}). Let us set
\[
U := \{ \varphi \in K \mid \ \|\varphi\| < \epsilon \ \}.
\]
Notice that there is no loss of generality in taking $\epsilon>0$ small enough that $\overline{U} \cap \overline{V} = \emptyset$ where $V$ is as in Lemma \ref{prop: V}.  Given $\varphi \in \overline{U}$, $x^\varphi(t) \to 0$ as $t \to \infty$; it is therefore clear that no power of $P$ has any fixed points on $\overline{U}$ other than $0\in K$.

Now choose $T > 0$ so large that $\kappa e^{\omega T} < \epsilon$, and let $n$ be an integer larger than $T/2$.  We claim that for any $\varphi \in \overline{U}$ we have $\|P^n(\varphi)\| < \epsilon$. In the situation $x^\varphi$ has fewer than $2n$ positive zeros, $P^n(\varphi) = 0$ by definition. Otherwise, the slowly oscillating behavior of $x^{\varphi}$ implies that the $2n$th positive zero $z$ of $x^\varphi$ occurs after time $2n - 1 > T - 1$ and thus
\[
\|P^n(\varphi)\| = \|x^\varphi_{z+1}\| = \| F(z+1,\varphi)\| \leq \kappa e^{\omega(z+1)}< \kappa e^{\omega T} < \epsilon.
\]
This completes the proof.
\end{proof}

Proposition \ref{prop: example one} now follows from Lemmas \ref{prop: V}, \ref{prop: U}, and the discussion at the beginning of the section.

\begin{remark}
Under the stronger assumption that $0 > f'(0) > -1$, it is possible to use much more elementary arguments to exhibit an open neighborhood $U$ in $K$ about $0$ such that $P(\overline{U}) \subset U$ and $\overline{U} \cap \overline{V} = \emptyset$.  In this case, we still have $i_{K}(P,U) = 1$ by property (I) of the index above, and the rest of the proof of Proposition \ref{prop: example one} is the same as before.  Here is a sketch of this simpler approach.

Since $f$ is $C^1$ on a neighborhood of $0$ and $|f'(0)| < 1$ there is an $\epsilon > 0$ such that $0 < |f'(\xi)| < 1$ for all $\xi \in [-\epsilon,\epsilon]$.  Set
\[
U := \{ \varphi \in K \mid \ \|\varphi\| < \epsilon \ \}.
\]
By shrinking the constant $\epsilon$ if necessary, we can guarantee that $\overline{U}$ and $\overline{V}$ are disjoint.

Now let $\varphi \in \overline{U}$ be given with $\varphi \neq 0$.  We claim that $\|P(\varphi)\| < \|\varphi\|$.  If $J(\varphi)  < 2$ with $J(\phi)$ introduced in Proposition \ref{prop: basic 2}, $P(\varphi) = 0$ and the claim is immediate.  Suppose that $J(\varphi) \geq 2$ and write $z_1 = z_1(\varphi)$, $z_2 = z_2(\varphi)$, and $x = x^\varphi$  Since $x(0) = \|\varphi\|$ and $x$ is decreasing on $[0,z_1]$, we have that $\|x_{z_1}\| \leq \|\varphi\|$.  For all $s \in [-1,0]$, since $|f'(\xi)| < 1$ for $\xi \in [-\epsilon,\epsilon]$ and $f(0)=0$ we have
\[
|f(x_{z_1}(s))| < |x_{z_1}(s)| \leq \|\varphi\|.
\]
Therefore it follows $0 > x'(t) > -\|\varphi\|$ for all $t \in (z_1,z_1+1)$, and $\|x_{z_1 + 1}\| < \|\varphi\|$.  A symmetric argument shows that $\|P(\varphi)\| < \|x_{z_1 + 1}\| < \|\varphi\|$.
\end{remark}

\section{Solutions of long period}

The interest of our next example for Eq. \eqref{eq: proto} with multiple slowly oscillating periodic solutions lies in the fact that we obtain solutions of ``long period" even though $f$ can be taken bounded by some fixed number: the non-monotonicity of $f$ is key.  Furthermore, in this example we will discuss the existence of multiple slowly oscillating periodic solutions for Eq. \eqref{eq: proto} both in the case where the zero solution is stable and in the case where it is unstable.

In order to specify the feedback function $f:\R\to\R$ in detail, we need positive reals $a,c,\delta,\gamma$ satisfying the following conditions:
\begin{itemize}
\item[$(i)$] $c < \min(a,\delta)$;\smallskip
\item[$(ii)$] $\gamma \geq 4a$ and $\gamma>\delta$;\smallskip
\item[$(iii)$] $\delta + \frac{c}{\delta} \gamma \leq a$.
\end{itemize}

The following result is obvious and shows  that these conditions are not vacuous.

\begin{prop}
Given any $a>0$ and any $\gamma > 4a$, there is some $\delta_* > 0$ such that, for any $\delta \in (0,\delta_*)$, there is a real $c_*(\delta) > 0$ such that all of conditions (i)--(iii) are satisfied whenever $c \in (0,c_*(\delta))$.
\end{prop}

Using the reals $a,c,\gamma,\delta>0$, we are now able to formulate the additional assumptions on $f:\R\to\R$ that will be essential throughout the present section:
\begin{equation}\label{eq: assumptions 3}\tag{H''}
  \left\{\begin{array}{l}\mbox{$f$ is odd and $|f| \leq \gamma$ on $[0,2a]$;}\\[0.2cm]
\mbox{$f(x)=\begin{cases}-\gamma,& \text{for }x\in[a,2a-c];\\
-\delta,& \text{for }x\in[2a,\gamma].\end{cases}$}\\
  \end{array} \right.
\end{equation}
Provided that, apart from (\ref{eq: assumptions 3}), the function $f$ also satisfies the imposed standard hypothesis (\ref{eq: assumptions}), Eq. \eqref{eq: proto} has a slowly oscillating periodic solution as we now show.

\begin{prop}\label{prop: example 2}
  Given reals $a,c,\gamma,\delta>0$ satisfying conditions (i)--(iii), suppose that (H) and (H'') hold. Then Eq. \eqref{eq: proto} has a slowly oscillating periodic solution with minimal period greater than $4$ and given by the length of an interval containing three consecutive zeros.
\end{prop}
\begin{proof}
The reader may find it helpful to refer to Fig. 2.

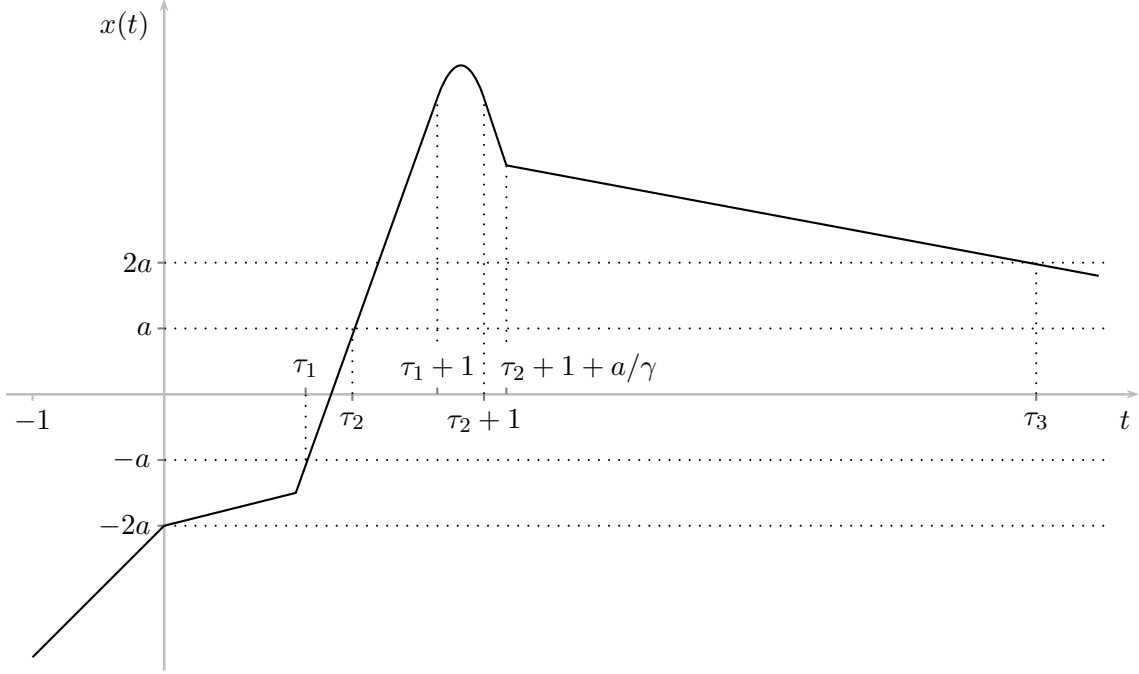
\begin{figure}[htb]
	\psset{unit=1.75cm}
	\begin{pspicture}(-1.5,-2.2)(7.5,3)
	\psaxes[labels=none,linecolor=lightgray,ticks=none]{->}(0,0)(-1.2,-2.1)(7.4,3)
	\psline[linecolor=lightgray](-1,0)(-1,-0.05)
	\psline[linecolor=black, linestyle=dotted](0,0.5)(7.2,0.5)
	\psline[linecolor=gray](0,0.5)(-0.05,0.5)
	\psline[linecolor=black, linestyle=dotted](0,-0.5)(7.2,-0.5)
	\psline[linecolor=gray](0,-0.5)(-0.05,-0.5)
	\psline[linecolor=black,linestyle=dotted](0,1.0)(7.2,1.0)
	\psline[linecolor=gray](0,1)(-0.05,1)
	\psline[linecolor=black,linestyle=dotted](0,-1.)(7.2,-1.)
	\psline[linecolor=gray](0,-1)(-0.05,-1)
	\rput(-0.25,-0.5){$-a$}
	\rput(-.15,0.5){$a$}
	\rput(-.2,1){$2a$}
	\rput(-0.3,-1){$-2a$}
	\psline(-1,-2)(0,-1)(1.,-0.75)(2.075,2.25)
	\pscurve(2.075,2.25)(2.255,2.5)(2.43,2.25)
	\psline(2.43,2.25)(2.6,1.74)
	\psline(2.6,1.74)(7.1,0.9)
	\psline[linecolor=black,linestyle=dotted](2.6,1.74)(2.6,0.4)
	\rput(-1.0,-0.2){$-1$}
	\psline[linecolor=black,linestyle=dotted](1.075,0)(1.075,-0.5)
	\psline[linecolor=gray](1.075,0)(1.075,0.05)
	\psline[linecolor=black,linestyle=dotted](2.075,0.4)(2.075,2.2)
	\psline[linecolor=gray](2.075,0)(2.075,0.05)
	\rput(1.075,0.2){$\tau_{1}$}
	\psline[linecolor=black,linestyle=dotted](1.43,0)(1.43,0.5)
	\psline[linecolor=gray](1.43,0)(1.43,-0.05)
	\rput(1.43,-0.2){$\tau_{2}$}
	\rput(2.075,.2){$\tau_{1}+1$}
	\psline[linecolor=gray](2.6,0)(2.6,0.05)
	\psline[linecolor=gray](2.43,0)(2.43,-0.05)
	\rput(2.43,-0.2){$\tau_{2}+1$}
	\psline[linecolor=black,linestyle=dotted](2.43,0)(2.43,2.2)
	\rput(3.15,0.2){$\tau_{2}+1+a/\gamma$}
	\psline[linecolor=black,linestyle=dotted](6.625,0)(6.625,1)
	\psline[linecolor=gray](6.625,0)(6.625,-0.05)
	\rput(6.625,-0.2){$\tau_{3}$}
	\rput(7.3,-0.2){$t$}
	\rput(-0.3,2.8){$x(t)$}
	\end{pspicture}
	\caption{Proof of Proposition \ref{prop: example 2}}\label{fig: proof}
\end{figure}

  1. Let $\varphi\in C$ be given with $\varphi(0)=-2a$ and with $-\gamma\leq \varphi(s)\leq -2a$ for all $s\in[-1,0]$, and let $x=x^{\varphi}:[-1,\infty)\to\R$ denote the uniquely determined solution of Eq. \eqref{eq: proto} with $x_{0}=\varphi$. Then, in view of assumption \eqref{eq: assumptions 3}, we have $x(t)=-2a+\delta t$ for all $0\leq t\leq 1$. In particular, $x$ is strictly increasing on $[0,1]$.

   Now, note that $c/\delta<1$ by condition (i) and $-2a+\delta<-a$ by condition (iii). Thus $x(t)\geq x(c/\delta)=-2a+c$ for $t\in[c/\delta,1]$ and $x(1)=-2a+\delta<-a$. Moreover, for all $t\in[1,1+c/\delta]$, $x(t)$ is increasing with derivative at most $\gamma$. We thus have the estimate
  \begin{equation*}
    x(1+c/\delta)\leq x(1)+\gamma\frac{c}{\delta} =-2a+\delta+\gamma\frac{c}{\delta},
  \end{equation*}
  which together with condition (iii) guarantees that $x(t)\in[-2a+c,-a]$ for all $t\in[c/\delta,1+c/\delta]$.

  2. Define $\tau_{1}:=\min\lbrace t\geq 1+c/\delta\mid x(t)=-a\rbrace$. Then the last observation in combination with assumption \eqref{eq: assumptions 3} shows that $x(\tau_{1}+s)=-a+\gamma\,s$ for all $s\in[0,1]$.  Set $\tau_{2}:=\tau_{1}+2a/\gamma$. By condition (ii) we have $\tau_{2} < \tau_2 + a/\gamma < \tau_{1}+1$, and so obtain
  \begin{equation*}
  x(\tau_{2})=a, \ \ x(\tau_2 + a/\gamma) = 2a, \ \ x(\tau_1 + 1) = -a + \gamma.
  \end{equation*}
 Thus, as $t$ increases from $\tau_{1}$ to $\tau_{1}+1$, $x(t)$ increases from the value $x(\tau_{1})=-a$ with constant slope $\gamma$, transverses the interval $[-a,a]$, and finishes at the value $x(\tau_{1}+1)=-a+\gamma$.  Note that $-a + \gamma \geq 3a$ by condition (ii).

  3. Now, from the oddness of the map $f$ and the fact that $[\tau_{1},\tau_{2}]\ni s\mapsto x(s)$ transverses the interval $[-a,a]$ at constant slope $\gamma$, it follows that $[0,2a/\gamma]\ni s\mapsto x(\tau_{1}+1+s)$
  traces out a symmetric arc with maximum at $s=a/\gamma$ and a maximal value of
  \begin{equation*}
  x(\tau_{1}+1+a/\gamma) < x(\tau_{1}+1)+\gamma \frac{a}{\gamma}=(-a+\gamma)+a=\gamma.
  \end{equation*}
   Furthermore, we have
  \begin{equation*}
    x(\tau_{2}+1)=x(\tau_{1}+2a/\gamma+1)=x(\tau_{1}+1)=-a+\gamma
  \end{equation*}
  and so $x(t)\in[3a,\gamma]$ for all $t\in[\tau_{1}+1,\tau_{2}+1]$.

 4. On the interval $[a,2a]$, $f$ is bounded below by $-\gamma$ but, by continuity, is not identically equal to  $-\gamma$. Accordingly, $[0,a/\gamma]\ni s\mapsto x(\tau_{2}+1+s)$ is decreasing from $x(\tau_{2}+1)=\gamma-a$ to $x(\tau_{2}+1+a/\gamma)$ where for this last value we crudely have
  \begin{equation*}
    x(\tau_{2}+1+a/\gamma)>x(\tau_{2}+1)-\gamma\frac{a}{\gamma}\geq -a+\gamma-a=\gamma-2a
  \end{equation*}
  and so $x(\tau_{2}+1+a/\gamma)>2a$ by condition (ii).  We have established that $x(t) \in (2a,\gamma)$ for all $t\in(\tau_{2}+a/\gamma,\tau_{2}+1+a/\gamma]$.

 5. As $t$ increases from $\tau_{2}+a/\gamma+1$,  $x(t)$ will decrease with slope $-\delta$ until time
  \begin{equation*}
  \tau_{3}:=\min\lbrace t>\tau_{2}+a/\gamma+1\mid x(t)=2a\rbrace>\tau_{2}+1+a/\gamma.
   \end{equation*}
   In particular, we have
$x(t)\in[2a,\gamma]$ for all $t\in[\tau_{3}-1,\tau_{3}]$. Thus $-x_{\tau_3}$ satisfies the same conditions as our initial condition $\varphi$; by using the oddness of $f$ and applying the arguments above again, it follows that $x(\tau_{3}+t)=-x(t)$ and secondly that
  \begin{equation*}
    x(2\tau_{3}+t)=-x(\tau_{3}+t)=-(-x(t))=x(t)
  \end{equation*}
  for all $t\geq 0$. In particular, since $x$ has only a single zero in $[0,\tau_{3}]$, the distance between any two zeros of $x$ in $[0,\infty)$ is greater than or equal to $\tau_{3}>1$. Hence, $\mathcal{O}:=\lbrace x(t)\mid 0\leq t\leq 2\tau_{3}\rbrace$ is the orbit of a slowly oscillating periodic solution of Eq. \eqref{eq: proto} with period  $2\tau_{3}$.

  5. It remains to prove that $\tau_{3}>2$. In order to do so, observe that
  \begin{equation*}
  \tau_{3}-(\tau_{2}+a/\gamma+1)
      \geq \frac{x(\tau_{2}+a/\gamma +1)-x(\tau_{3})}{\delta}
      >\frac{(\gamma-2a)-2a}{\delta}=\frac{\gamma-4a}{\delta}.
  \end{equation*}
Hence, it follows that
  \begin{equation}\label{eq: auxi_3}
    \tau_{3}>\tau_{2}+\frac{a}{\gamma}+1+\frac{\gamma-4a}{\delta}\geq \tau_{1}+\frac{2a}{\gamma}+\frac{a}{\gamma}+1+\frac{\gamma-4a}{\delta}\geq 2+\frac{c}{\delta}+\frac{3a}{\gamma}+\frac{\gamma-4a}{\delta}.
  \end{equation}
  This shows that $2\tau_{3}>4$ and completes the proof.
\end{proof}

\begin{remark}
 Observe that the last step of the proof indicates that, by choosing appropriate parameters $a,c, \gamma,\delta>0$ and adapting the map $f$ correspondingly, the period $2\tau_{3}$ of the resulting slowly oscillating periodic solution can be made as large as desired. For since
  \begin{equation*}
    \tau_{3}>2+\frac{\gamma-4a}{\delta}
  \end{equation*}
  by estimate \eqref{eq: auxi_3}, we can make $2\tau_{3}$ as large as we like by taking $\gamma > 4a$ and then choosing $\delta$ (and then $c$) small enough given parameters $a$ and $\gamma$.
\end{remark}

Of course, the slowly oscillating periodic solution found in Proposition \ref{prop: example 2} has a segment in the closed bounded convex set $K\subset C$ that is a nonzero fixed point of the return map $P$. Moreover, this periodic solution is so strongly attractive that $P$ is constant in each sufficiently small neighborhood of this fixed point, as we now show.

\begin{prop}\label{prop: V2}
  Under the hypothesis of Proposition \ref{prop: example 2}, there is an open convex subset $V\subset K$ such that $0\not\in\overline{V}$ and $P(\overline{V})\subset V$.

Moreover, the fixed point of $P$ in $V$ is a segment of the slowly oscillating periodic solution of Eq. \eqref{eq: proto} obtained in Proposition \ref{prop: example 2} and $P$ maps all of $\overline{V}$ to this fixed point.
\end{prop}
\begin{proof}
We revisit the proof of Proposition \ref{prop: example 2} and use the same notation in the following.
To begin with, define $s_{1}:=\tau_{1}+a/\gamma+1$ and $s_{2}:=\tau_{2}+a/\gamma+1$.  We have $x_{s_{1}}\in K$.

Now, the proof of Proposition \ref{prop: example 2} shows that there are numbers $\triangle > 0$ and $\epsilon > 0$ such that $x(t) \in (2a + \epsilon,\gamma - \epsilon)$ for all $t \in [s_2+\triangle - 1, s_2 + \triangle]$.  Now by continuous dependence on initial conditions there is some open set $V$ about $x_{s_1}$ such that $y_0 \in \overline{V}$ implies that $y(s) \in (2a,\gamma)$ for $s \in [s_2+\triangle - 1- s_1, s_2 + \triangle - s_1]$.  The definition of $f$ now allows $y$ to continue in only one way; in particular, $y(s)$ will coincide with a translate of $x$ for $s \geq s_2 + \triangle - s_1$.
\end{proof}

Note that the stability of the zero solution of \eqref{eq: proto} has been, so far, completely irrelevant for our discussion of the current example. On the other hand, under the assumption of local exponential stability of the zero solution we can take the same approach as in Lemma \ref{prop: U} to show the existence of multiple slowly oscillating periodic solutions of Eq. \eqref{eq: proto}.  We accordingly omit the proof of the following proposition.

\begin{prop}\label{prop: example 2 stable}
  Under the hypotheses of Proposition \ref{prop: example 2}, if moreover $f^{\prime}(0)\in(-\pi/2,0)$ holds then Eq. \eqref{eq: proto} has at least two distinct slowly oscillating periodic solutions, each with period given by three consecutive zeros.
\end{prop}

But also in the situation where $f^{\prime}(0)<-\pi/2$ holds, and thus the zero solution of Eq. \eqref{eq: proto} is unstable, Eq. \eqref{eq: proto} has at least two distinct slowly oscillating periodic solutions, provided that the conditions of Proposition \ref{prop: example 2} are satisfied and that (for example) $f$ is constant outside some neighborhood of the origin --- for then Theorem \ref{KYTHM} applies.  We therefore have the following.

\begin{prop}\label{prop: example 2 unstable period 4}
  Suppose that the assumptions of Proposition \ref{prop: example 2} are satisfied, that $f$ is constant outside some open neighborhood of $0\in\R$, and that additionally $f^{\prime}(0)<-\pi/2$. Then Eq. \eqref{eq: proto} has at least two distinct slowly oscillating periodic solutions.

  More precisely, apart from the slowly oscillating periodic solution with minimal period greater than $4$, Eq. \eqref{eq: proto} has a slowly oscillating periodic solution $x:\R\to\R$ with minimal period $4$ and the symmetry $x(t)=-x(t-2)$ for all $t\in\R$.
\end{prop}

\begin{remark}
  The last statement of course remains true if the assumption that $f$ is constant outside some open neighborhood of the origin is replaced by any condition that makes Theorem \ref{KYTHM} hold.
\end{remark}

It is now natural to ask whether, at least under additional conditions on $f$, one of the slowly oscillating periodic solutions obtained in Proposition \ref{prop: example 2 stable} (and so in case of a locally asymptotically stable zero solution) may also be a Kaplan-Yorke solution.  In order to address this issue, we should recall some details of the approach taken in \cite{Kaplan and Yorke 1974} for the proof of the existence of Kaplan-Yorke solutions for Eq. \eqref{eq: proto}.

We continue to assume that $f$ satisfies assumptions (H) and (H''), is globally $C^{1}$-smooth, and is constant outside some neighborhood of the origin. As under these conditions we have $H(u,v)\to\infty$ as $u^2+v^2\to\infty$ for the Hamiltonian $H$ defined by Eq. \eqref{eq: Hamiltonian} it follows that for each initial condition of the form $(u_{0},0)\in\R^{2}$, where $u_{0}>0$, Eq. \eqref{eq: ODE} has a uniquely determined periodic solution $r(\cdot;u_{0})=(u(\cdot;u_{0}),v(\cdot;u_{0})):\R\to\R^{2}$ with $r(0;u_{0})=(u_{0},0)$. Its orbit is given by $H^{-1}(\lbrace\alpha\rbrace)$ for $\alpha=-\int_{0}^{u_{0}}f(s)\,ds$. Set now for each $u_{0}>0$
 \begin{equation}\label{eq: period}
   \tau(u_{0}):=\min\lbrace t>0\mid u(t;u_{0})=0\rbrace,
 \end{equation}
 which is well-defined since $r(\cdot;u_{0})$ traces out a periodic orbit around the origin. Moreover, in view of $r^{\prime}(0;u_{0})=(0,-f(u_{0}))$ with $f(u_{0})<0$ and the discussed symmetry of $H$, if follows that
 \begin{align*}
 r(0;u_{0})=(u_{0},0),\quad r(\tau(u_{0});u_{0})=(0,u_{0}),\quad
 r(2\tau(u_{0});u_{0})
   =(-u_{0},0),\quad
   r(3\tau(u_{0});u_{0})=(0,-u_{0})
\end{align*}
and
\begin{equation*}
r(4\,\tau(u_{0});u_{0})=(u_{0},0)=r(0;u_{0}).
\end{equation*}
In particular, the solution $r(\cdot;u_{0})$ has the period $4\tau(u_{0})$ and this period is minimal. Accordingly, in order to find a periodic solution of Eq. \eqref{eq: proto} we seek $u_{0}>0$ such that $\tau(u_{0})=1$. The main tool for doing so will be the following lemma (see \cite{Kaplan and Yorke 1974}).

\begin{lem}\label{lem: period cont}
  Assume that $f:\R\to\R$ is $C^{1}$-smooth, odd, satisfies the negative feedback condition, and is constant outside some neighborhood of the origin. Then the definition \eqref{eq: period} induces a continuous function $\tau:(0,\infty)\to(0,\infty)$ with
  \begin{equation*}
  \lim_{u\to 0^{+}}\tau(u)=-\frac{\pi}{2}\frac{1}{f^{\prime}(0)}\qquad\text{and}\qquad\lim_{u\to\infty}\tau(u)=\infty.
  \end{equation*}
\end{lem}

\begin{proof}
    Since the proof of the continuity of $\tau$ in \cite{Kaplan and Yorke 1974} is somewhat cursory, we give a fuller proof. Let $u_{0}>0$ and $\epsilon>0$ be given. Set $\tau_{0}:=\tau(u_{0})$ with $\tau(u_{0})$ defined by formula \eqref{eq: period}. As $r(\tau_{0};u_{0})=(0,u_{0})$ and $r(\cdot;u_{0})$ intersects each half-axis transversally, there clearly is some $0<\tau_{1}<\min\lbrace\tau_{0},\epsilon/2\rbrace$ such that $u(\tau_{0}+t;u_{0})<0<u(\tau_{0}-t;u_{0})$
     for all $0< t\leq \tau_{1}$. Fix now any $0<\epsilon_{1}<\epsilon$ satisfying both
     \begin{equation*}
       \epsilon_{1}<\min\lbrace |u(\tau_{0}-\tau_{1};u_{0})|,|u(\tau_{0}+\tau_{1};u_{0})|\rbrace\quad \text{and}\quad
       \epsilon_{1}<\min\lbrace \|r(t;u_{0})\|_{\R^{2}}\mid 0\leq t\leq t+\tau_{1}\rbrace
     \end{equation*}
     where $\|\cdot\|_{\R^{2}}$ denotes the Euclidean norm in $\R^{2}$.
      Then by elementary results on continuous dependence of solutions for ordinary differential equation we find some $\delta>0$ such that for each $\tilde{u}>0$ with $|\tilde{u}-u_{0}|<\delta$ and all $0\leq t\leq \tau_{0}+\tau_{1}$ we have
  \begin{equation*}
    \max\left\lbrace |u(t;\tilde{u})-u(t;u_{0})|,|v(t;\tilde{u})-v(t;u_{0})|\right\rbrace\leq\|r(t;u_{0}) -
    r(t,\tilde{u})\|_{\R^{2}}<\epsilon_{1}.
  \end{equation*}
  Consider now any $\tilde{u}> 0$ with $|\tilde{u}-u_{0}|<\delta$ and observe that on the one hand we have
  \begin{equation*}
    \begin{aligned}
      u(\tau_{0}+\tau_{1};\tilde{u})&=u(\tau_{0}+\tau_{1};\tilde{u})-u(\tau_{0}+\tau_{1};u_{0})+
      u(\tau_{0}+\tau_{1};u_{0})\\
      &\leq \|r(\tau_{0}+\tau_{1};\tilde{u})-r(\tau_{0}+\tau_{1};u_{0})\|_{\R^{2}}+u(\tau_{0}+\tau_{1};u_{0})\\
      &<\epsilon_{1}+u(\tau_{0}+\tau_{1};u_{0})\\
      &\leq \epsilon_{1}-\epsilon_{1}\\
      &=0
    \end{aligned}
  \end{equation*}
  and on the other hand
  \begin{equation*}
    \begin{aligned}
      u(\tau_{0}-\tau_{1};\tilde{u})&=u(\tau_{0}-\tau_{1};\tilde{u})-u(\tau_{0}-\tau_{1};u_{0})+
      u(\tau_{0}-\tau_{1};u_{0})\\
      &>\epsilon_{1}-|u(\tau_{0}-\tau_{1};\tilde{u})-u(\tau_{0}-\tau_{1};u_{0})|\\
      &\geq \epsilon_{1}-\epsilon_{1}\\
      &=0.
    \end{aligned}
  \end{equation*}
  Hence, the continuity of solution $r(\cdot;\tilde{u})$ and so of its first component $u(\cdot;\tilde{u})$ implies the existence of some $\tau_{0}-\tau_{1}<T(\tilde{u})<\tau_{0}+\tau_{1}$ with $u(T(\tilde{u});\tilde{u})=0$.

 Suppose now that $T(\tilde{u})\not=\min\lbrace t>0\mid u(t;\tilde{u})=0\rbrace$; that is, suppose that there is some $0<\xi<T(\tilde{u})$ with $u(\xi;\tilde{u})=0$. Then our discussion about the solutions of Eq. \eqref{eq: ODE} would imply the existence of some $0<\zeta<T(\tilde{u})$ with $r(\zeta;\tilde{u})=(0,-\tilde{u})$. But in view of
  \begin{equation*}
    \begin{aligned}
      \epsilon_{1}^{2}&>\|r(\zeta;u_{0})-r(\zeta;\tilde{u})\|_{\R^{2}}^{2}\\
      &=(u(\zeta;u_{0})-u(\zeta;\tilde{u}))^{2}+(v(\zeta;u_{0})-v(\zeta;\tilde{u}))^{2}\\
      &=(u(\zeta;u_{0}))^{2}+(v(\zeta;u_{0})+\tilde{u})^{2}\\
      &\geq \|r(\zeta;u_{0})\|_{\R^{2}}^{2}+\tilde{u}^{2}\\
      &\geq \epsilon_{1}^{2}+\tilde{u}^{2}
    \end{aligned}
  \end{equation*}
  that is impossible. Therefore, we have $T(\tilde{u})=\tau(\tilde{u})$ where the right-hand side is defined by \eqref{eq: period}. Moreover, it follows that
  \begin{equation*}
    |\tau(u_{0})-\tau(\tilde{u})|=|\tau_{0}-T(\tilde{u})|<\tau_{1}<\frac{\epsilon}{2}<\epsilon
  \end{equation*}
  and this proves the continuity of the function $\tau:(0,\infty)\to(0,\infty)$ defined by \eqref{eq: period}.

  The assertions about the behavior of the function $\tau$ for $t\to 0^{+}$ and $t\to\infty$ are proven in  Theorem 1 of \cite{Kaplan and Yorke 1974}.
\end{proof}

\begin{remark}
  Under the assumptions of the last statement it is possible to prove that $\tau:(0,\infty)\to(0,\infty)$ is not only continuous but continuously differentiable.  This follows from straightforward application of elementary results about smooth dependence of solutions of ordinary differential equations on initial conditions, along with the implicit function theorem.
\end{remark}

Now note that if the assumptions on $f$ from the last lemma are satisfied, and if moreover $f^{\prime}(0)<-\pi/2$ and so the zero solution of Eq. \eqref{eq: proto} is unstable, it follows that $\tau(u)<1$ for all sufficiently small $u>0$ whereas $\tau(u)>1$ for all sufficiently large $u>0$. Thus, the intermediate value theorem implies the existence of some $0<u_{0}<\infty$ with $\tau(u_{0})=1$, and so Eq. \eqref{eq: proto} has at least one Kaplan-Yorke solution.

However, the same reasoning does not work in the situation $-\pi/2<f^{\prime}(0)<0$, that is, when the zero solution of Eq. \eqref{eq: proto} is locally asymptotically stable. Indeed, in this case Lemma \ref{lem: period cont} implies that $\tau(u)>1$ for all sufficiently small as well as for all sufficiently large $u>0$, and so a straightforward application of the intermediate value theorem fails. In our next result we will give conditions on $f$ similar to those of Proposition \ref{prop: example 2 unstable period 4} but without any restrictions on $f^{\prime}(0)<0$ such that the intermediate value theorem will be applicable regardless of the stability properties of the zero solution of Eq. \eqref{eq: proto}. In doing so, we shall refer to Fig. \ref{fig: proof2}, which describes the value of the vector field for Eq. \eqref{eq: ODE} with $f$ satisfying (H'') in various regions of the first quadrant of the real plane.

\begin{figure}[htb]
\psset{unit=1.75cm}
\begin{pspicture}(-.3,-.3)(5.2,5.2)
\psaxes[labels=none,linecolor=lightgray,ticks=none]{->}(0,0)(-0.1,-0.1)(5.1,5.1)
\psline[linestyle=dotted](0,1.5)(5,1.5)
\psline[linestyle=dotted](1.5,0)(1.5,5)
\psline[linecolor=gray](1.5,0)(1.5,-0.05)
\psline[linecolor=gray](0,1.5)(-0.05,1.5)
\rput(-.15,1.5){$a$}
\rput(1.5,-.15){$a$}
\rput(5,-0.15){$u$}
\rput(-0.15,5){$v$}
\psline[linestyle=dotted](2.75,0)(2.75,5)
\psline[linestyle=dotted](0,2.75)(5,2.75)
\psline[linecolor=gray](2.75,0)(2.75,-0.05)
\psline[linecolor=gray](0,2.75)(-0.05,2.75)
\rput(2.75,-0.15){$2a-c$}
\rput(-0.4,2.75){$2a-c$}
\rput(2.125,0.75){$(f(v),\gamma)$}
\rput(0.75,2.125){$(-\gamma,-f(u))$}
\rput(2.125,2.125){$(-\gamma,\gamma)$}
\psline[linestyle=dotted](3.5,0)(3.5,5)
\psline[linestyle=dotted](0,3.5)(5,3.5)
\psline[linecolor=gray](-0.05,3.5)(0,3.5)
\psline[linecolor=gray](3.5,-0.05)(3.5,0)
\rput(4.25,4.25){$(-\delta,\delta)$}
\rput(3.5,-0.15){$2a$}
\rput(-0.2,3.5){$2a$}
\rput(4.25,0.75){$(f(v),\delta)$}
\rput(0.75,4.25){$(-\delta,-f(u))$}
\rput(4.25,2.125){$(-\gamma,\delta)$}
\rput(2.125,4.25){$(-\delta,\gamma)$}
\end{pspicture}
\caption{Value of the vector field for Eq. \eqref{eq: ODE} in various regions for $u,v\geq0$}\label{fig: proof2}
\end{figure}
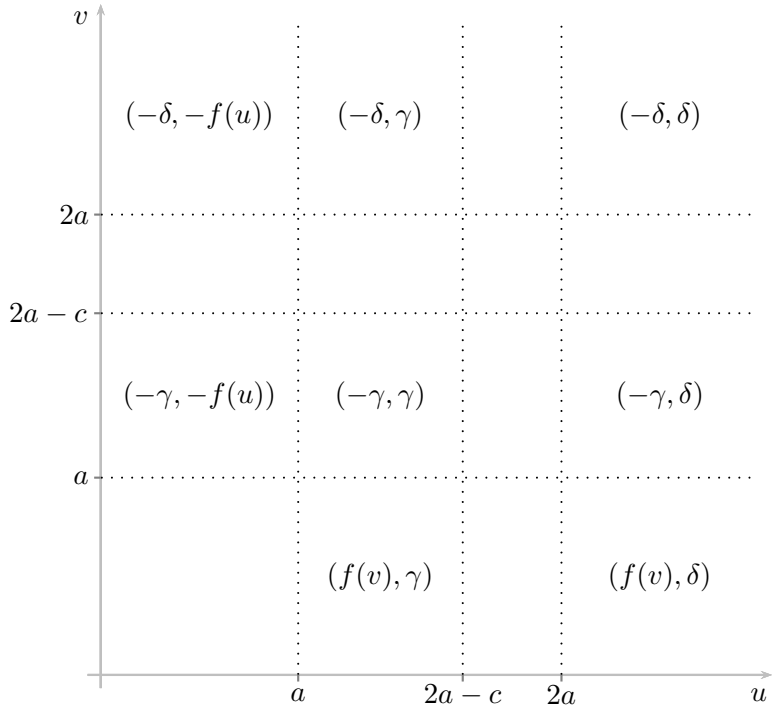

\begin{prop}\label{prop: app inter}
  Given reals $a,c,\gamma,\delta>0$ and satisfying conditions (i)--(iii), assume that $f$ is $C^{1}$-smooth, satisfies both (H) and (H''), and is constant outside some open neighborhood of $0\in\R$. If
  \begin{equation}\label{eq: condition2}
  \frac{1}{\gamma}\int_{0}^{a}|f(s)|\,ds<a-c
  \end{equation}
  then $\tau(2a-c)<1$ with $\tau(2a-c)$ defined by \eqref{eq: period}.
\end{prop}
\begin{proof}
  Under the given assumptions, consider the periodic solution $r(\cdot;2a-c):\R\to\R^{2}$ of Eq. \eqref{eq: ODE} with initial condition $r(0;2a-c)=(2a-c,0)$. We have
  \begin{equation*}
    r^{\prime}(0;2a-c)=\left(u^{\prime}(0;2a-c), v^{\prime}(0;2a-c)\right)=\left(
      f(0), -f(2a-c)
    \right)=\left(
      0, \gamma
    \right)
  \end{equation*}
  and $u^{\prime}(t;2a-c)<0$ for all sufficiently small $t>0$. Hence, as $t$ moves forward, as long as $a\leq u(t)\leq 2a-c$ and $0\leq v(t)\leq a$, we will have
  \begin{equation*}
    \left(
      u^{\prime}(t;2a-c),v^{\prime}(t;2a-c)
    \right)=\left(
      f(v(t;2a-c)), \gamma
    \right).
  \end{equation*}
  Thus, for all such $t$, we actually have $v(t;2a-c)=\gamma\,t$ and
  \begin{equation*}
    u(t;2a-c)=u(0;2a-c)+\int_{0}^{t}f(\gamma\,s)\,ds=2a-c+\frac{1}{\gamma}\int_{0}^{\gamma t}f(s)\,ds.
  \end{equation*}
  Now condition \eqref{eq: condition2} implies that $a\leq u(t;2a-c)\leq 2a-c$ for all $0\leq t\leq a/\gamma$, and so at precisely $t=a/\gamma$ we have
  \begin{equation*}
    u(t;2a-c)=2a-c+\frac{1}{\gamma}\int_{0}^{a}f(s)\,ds\in [a,2a-c]\quad\text{and}\quad v(t;2a-c)=a.
  \end{equation*}
  The solution $r(\cdot;2a-c)$ will now move diagonally across the ``box'' $[a,2a-c]\times[a,2a-c]$, with constant velocity
  \begin{equation*}
  r^{\prime}(t;2a-c)=\left(
    u^{\prime}(t;2a-c), v^{\prime}(t;2a-c)
  \right)=\left(
    -\gamma, \gamma
  \right).
  \end{equation*}
  This transversal will clearly take fewer than $(a-c)/\gamma$ time units. Then, by symmetry, the solution $r(\cdot;2a-c)$ will take another $a/\gamma$ units to reach $r(t;2a-c)=(0,2a-c)$. Therefore we have
  \begin{equation*}
    \tau(2a-c) < \frac{a+(a-c)+a}{\gamma}<\frac{3a}{\gamma}<1
  \end{equation*}
  by condition (ii), and this finishes the proof.
\end{proof}

A direct consequence of the last result is now the following corollary.
\begin{cor}\label{cor: maxvalues}
  If the hypotheses on $f$ described in Proposition \ref{prop: app inter} (condition \eqref{eq: condition2} included) hold, then Eq. \eqref{eq: proto} has a Kaplan-Yorke solution with maximal value larger than $2a-c$.

  If furthermore $f^{\prime}(0) > -\pi/2$, then Eq. \eqref{eq: proto} also has a Kaplan-Yorke solution with maximum value less than $2a-c$.
\end{cor}
\begin{proof}
  Under the given conditions, Lemma \ref{lem: period cont} together with the last proposition show that for the continuous map $\tau:(0,\infty)\to (0,\infty)$ defined by \eqref{eq: period} we have $\tau(2a-c)<1$ whereas $\tau(u)>1$ for all sufficiently large $u>2a-c$. Hence, there clearly exists some $2a-c<u_{0}<\infty$ with $\tau(u_{0})=1$ such that $x:\R\to\R$ with $x(t):=u(t;u_{0})$ for $t\in\R$ forms a Kaplan-Yorke solution of Eq. \eqref{eq: proto} with $\max_{t\in\R}x(t)=u_{0}>2a-c$, as claimed.

Under the additional assumption $f^{\prime}(0)> -\pi/2$, Lemma \ref{lem: period cont} also proves that $\tau(u)>1$ for all sufficiently small $u>0$. Hence, apart from $2a-c<u_{0}<\infty$, there is some $0<\tilde{u}_{0}<2a-c$ with $\tau(\tilde{u}_{0})=1$. If follows that $\tilde{x}:\R\to\R$ defined by $\tilde{x}(t):=u(t;\tilde{u}_{0})$ for all $t\in\R$ is as well a Kaplan-Yorke solution of Eq. \eqref{eq: proto} but with maximum value $\max_{t\in\R}\tilde{x}(t)=\tilde{u}_{0}<2a-c$.
\end{proof}

Let us consider an example.  We take $f$ as described before Proposition \ref{prop: example 2}, with $\gamma =4$, $a = 1$, $\delta = 2/3$, and $c = 1/20$.  Note that conditions $(i)$ -- $(iii)$ are satisfied.  In addition to hypotheses (\ref{eq: assumptions 3}), we also take $f(x) = -2x$ on $[-a+c,a-c]$ (so the zero solution is unstable).  Now, Proposition \ref{prop: app inter} shows that Eq. \eqref{eq: proto} will have a Kapan-Yorke solution $q$ with maximum value greater than $2a-c$.   On the other hand, in this case a solution $r = (u,v)$ of Eq. \eqref{eq: ODE} with initial condition $(3a,0) = (3,0)$ will have period strictly greater than $4$.  Referring once more to Fig. \ref{fig: proof2}, we see that that as long as $(u(t),v(t))$ is in the ``box'' bounded by the lines $y = 0$, $y = a-c$, $x = 2a$, and $x = \gamma$, we will have $v'(t) = \delta$ and $v(t) = \delta t = 2t/3$, and
\[
u'(t) = f(v(t)) = -2v(t) = -2\delta t, \ \ u(t) = 3 - \delta t^2 = 3 - \frac{2}{3}t^2.
\]
Since $c = 1/20$, we see that $r(1)$ is still in the above-described ``box,'' and hence that $r$ has period strictly greater than $4$.  We conclude that our Kaplan-Yorke solution $q$ has a trace in the $(x(t),x(t-1))$ plane going through the points $(\alpha,0)$ and $(0,\alpha)$, where $2a - c < \alpha < 3a$.

On the other hand, the proof of Proposition \ref{prop: example 2} shows that the slowly oscillating periodic solution $p$ found there satisfies $|p'(z)| = \gamma$ and $|p(z-1)| \in [a,2a-c]$ whenever $p(z) = 0$, and attains a maximum value of greater than $3a$.   Thus the trace of $p$ in the $(x(t),x(t-1))$ plane is a simple closed curve going through the points $(\alpha_2,0)$ and $(0,\alpha_1)$, where $\alpha_1 < \alpha < \alpha_2$.  Thus the traces of $q$ and $p$ intersect in the $(x(t),x(t-1))$ plane; such an intersection of the traces of two SOP solutions is not possible when $f$ is monotonic (again, see \cite{Kaplan and Yorke 1975}).

The figures below show numerical approximations of two solutions $q$ and $p$ as described just above.  The first figure shows the graphs of the two solutions; the second shows the solutions' traces in the plane.  (In making the approximations, we have taken  $f$ linear on the intervals $[a-c,a]$ and $[2a-c,2a]$, and so there are isolated points where $f$ is not smooth.)
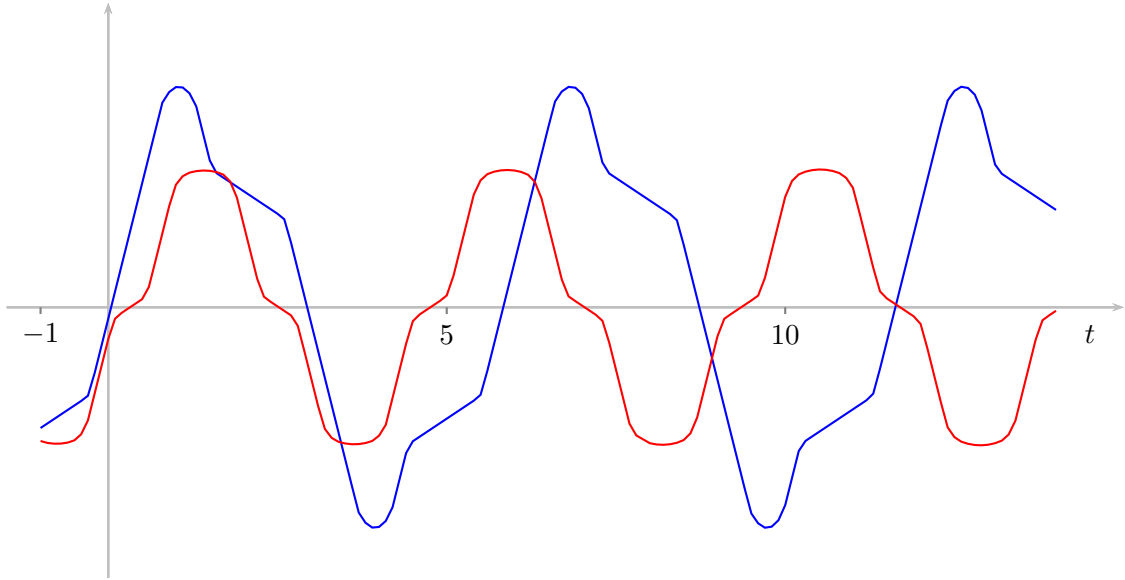
\begin{figure}[htb]
	\centering
	\psset{unit=0.9cm}
	\begin{pspicture}(-2,-4.)(15,5.)
	\psaxes[labels=none,linecolor=lightgray,ticks=none]{->}(0,0)(-1.5,-4)(15.,4.5)
	\rput(-1.,-0.4){$-1$}
	\psline[linecolor=gray](-1.0,0)(-1.0,-0.1)
	\rput(5,-0.4){$5$}
	\psline[linecolor=gray](5.0,0)(5.0,-0.1)
	\rput(10,-0.4){$10$}
	\psline[linecolor=gray](10,0)(10,-0.1)
	\rput(14.5,-0.4){$t$}
	
	\psline[linecolor=blue](-1,-1.78522)
	(-0.9,-1.71855)
	(-0.8,-1.65188)
	(-0.7,-1.58522)
	(-0.6,-1.51855)
	(-0.5,-1.45188)
	(-0.4,-1.38522)
	(-0.3,-1.30767)
	(-0.2,-0.9698)
	(-0.1,-0.5698)
	(0,-0.1698)
	(0.1,0.2302)
	(0.2,0.6302)
	(0.3,1.0302)
	(0.4,1.4302)
	(0.5,1.8302)
	(0.6,2.2302)
	(0.7,2.6302)
	(0.8,3.02548)
	(0.9,3.18148)
	(1,3.25547)
	(1.1,3.24947)
	(1.2,3.16347)
	(1.3,2.9702)
	(1.4,2.5702)
	(1.5,2.1702)
	(1.6,1.9827)
	(1.7,1.91603)
	(1.8,1.84937)
	(1.9,1.7827)
	(2,1.71603)
	(2.1,1.64937)
	(2.2,1.5827)
	(2.3,1.51603)
	(2.4,1.44937)
	(2.5,1.3827)
	(2.6,1.30113)
	(2.7,0.9547)
	(2.8,0.5547)
	(2.9,0.1547)
	(3,-0.2453)
	(3.1,-0.6453)
	(3.2,-1.0453)
	(3.3,-1.4453)
	(3.4,-1.8453)
	(3.5,-2.2453)
	(3.6,-2.6453)
	(3.7,-3.03462)
	(3.8,-3.18572)
	(3.9,-3.2567)
	(4,-3.24768)
	(4.1,-3.15866)
	(4.2,-2.9551)
	(4.3,-2.5551)
	(4.4,-2.1551)
	(4.5,-1.98018)
	(4.6,-1.91352)
	(4.7,-1.84685)
	(4.8,-1.78018)
	(4.9,-1.71352)
	(5,-1.64685)
	(5.1,-1.58018)
	(5.2,-1.51352)
	(5.3,-1.44685)
	(5.4,-1.38018)
	(5.5,-1.29395)
	(5.6,-0.9396)
	(5.7,-0.5396)
	(5.8,-0.1396)
	(5.9,0.2604)
	(6,0.6604)
	(6.1,1.0604)
	(6.2,1.4604)
	(6.3,1.8604)
	(6.4,2.2604)
	(6.5,2.6604)
	(6.6,3.04189)
	(6.7,3.18985)
	(6.8,3.25781)
	(6.9,3.24577)
	(7,3.15373)
	(7.1,2.94)
	(7.2,2.54)
	(7.3,2.14)
	(7.4,1.97767)
	(7.5,1.911)
	(7.6,1.84433)
	(7.7,1.77767)
	(7.8,1.711)
	(8,1.57767)
	(8.1,1.511)
	(8.2,1.44433)
	(8.3,1.37767)
	(8.4,1.28614)
	(8.5,0.9245)
	(8.6,0.5245)
	(8.7,0.1245)
	(8.8,-0.2755)
	(8.9,-0.6755)
	(9,-1.0755)
	(9.1,-1.4755)
	(9.2,-1.8755)
	(9.3,-2.2755)
	(9.4,-2.6755)
	(9.5,-3.04893)
	(9.6,-3.19387)
	(9.7,-3.25881)
	(9.8,-3.24375)
	(9.9,-3.14869)
	(10,-2.9249)
	(10.1,-2.5249)
	(10.2,-2.1249)
	(10.3,-1.97515)
	(10.4,-1.90848)
	(10.5,-1.84182)
	(10.6,-1.77515)
	(10.7,-1.70848)
	(10.8,-1.64182)
	(10.9,-1.57515)
	(11,-1.50848)
	(11.1,-1.44182)
	(11.2,-1.37515)
	(11.3,-1.27769)
	(11.4,-0.9094)
	(11.5,-0.5094)
	(11.6,-0.1094)
	(11.7,0.2906)
	(11.8,0.6906)
	(11.9,1.0906)
	(12,1.4906)
	(12.1,1.8906)
	(12.2,2.2906)
	(12.3,2.6906)
	(12.4,3.05586)
	(12.5,3.19778)
	(12.6,3.2597)
	(12.7,3.24162)
	(12.8,3.14354)
	(12.9,2.9098)
	(13,2.5098)
	(13.1,2.1098)
	(13.2,1.97263)
	(13.3,1.90597)
	(13.4,1.8393)
	(13.5,1.77263)
	(13.6,1.70597)
	(13.7,1.6393)
	(13.8,1.57263)
	(13.9,1.50597)
	(14,1.4393)
	\psline[linecolor=red](-1,-1.9772)
	(-0.9,-2.00489)
	(-0.8,-2.01679)
	(-0.7,-2.01536)
	(-0.6,-2.00059)
	(-0.5,-1.96708)
	(-0.4,-1.87722)
	(-0.3,-1.67191)
	(-0.2,-1.27191)
	(-0.1,-0.87191)
	(0,-0.47191)
	(0.1,-0.16963)
	(0.2,-0.08139)
	(0.3,-0.01473)
	(0.4,0.05194)
	(0.5,0.11885)
	(0.6,0.30055)
	(0.7,0.69436)
	(0.8,1.09436)
	(0.9,1.49436)
	(1,1.81227)
	(1.1,1.93746)
	(1.2,1.99037)
	(1.3,2.01275)
	(1.4,2.02153)
	(1.5,2.01699)
	(1.6,1.99911)
	(1.7,1.96063)
	(1.8,1.86099)
	(1.9,1.62236)
	(2,1.22236)
	(2.1,0.82236)
	(2.2,0.42236)
	(2.3,0.15587)
	(2.4,0.08009)
	(2.5,0.01342)
	(2.6,-0.05325)
	(2.7,-0.11991)
	(2.8,-0.27533)
	(2.9,-0.66049)
	(3,-1.06049)
	(3.1,-1.46049)
	(3.2,-1.79836)
	(3.3,-1.93046)
	(3.4,-1.9885)
	(3.5,-2.01348)
	(3.6,-2.02476)
	(3.7,-2.02271)
	(3.8,-2.00733)
	(3.9,-1.97643)
	(4,-1.89851)
	(4.1,-1.73662)
	(4.2,-1.3382)
	(4.3,-0.9382)
	(4.4,-0.5382)
	(4.5,-0.20444)
	(4.6,-0.10591)
	(4.7,-0.03924)
	(4.8,0.02742)
	(4.9,0.09409)
	(5,0.17481)
	(5.1,0.46647)
	(5.2,0.86647)
	(5.3,1.26647)
	(5.4,1.66647)
	(5.5,1.8768)
	(5.6,1.96818)
	(5.7,2.00462)
	(5.8,2.02366)
	(5.9,2.02936)
	(6,2.02173)
	(6.1,2.00077)
	(6.2,1.95972)
	(6.3,1.85725)
	(6.4,1.60942)
	(6.5,1.20942)
	(6.6,0.80942)
	(6.7,0.40942)
	(6.8,0.16097)
	(6.9,0.0906)
	(7,0.02393)
	(7.1,-0.04273)
	(7.2,-0.1094)
	(7.3,-0.19973)
	(7.4,-0.5214)
	(7.5,-0.9214)
	(7.6,-1.3214)
	(7.7,-1.7214)
	(7.8,-1.89449)
	(8,-2.0096)
	(8.1,-2.02803)
	(8.2,-2.03312)
	(8.3,-2.02488)
	(8.4,-2.00331)
	(8.5,-1.9631)
	(8.6,-1.86409)
	(8.7,-1.62769)
	(8.8,-1.22769)
	(8.9,-0.82769)
	(9,-0.42769)
	(9.1,-0.17031)
	(9.2,-0.09949)
	(9.3,-0.03282)
	(9.4,0.03384)
	(9.5,0.10051)
	(9.6,0.17191)
	(9.7,0.43299)
	(9.8,0.83299)
	(9.9,1.23299)
	(10,1.63299)
	(10.1,1.86639)
	(10.2,1.96452)
	(10.3,2.00503)
	(10.4,2.02756)
	(10.5,2.03676)
	(10.6,2.03263)
	(10.7,2.01517)
	(10.8,1.98368)
	(10.9,1.91321)
	(11,1.76583)
	(11.1,1.38585)
	(11.2,0.98585)
	(11.3,0.58585)
	(11.4,0.2352)
	(11.5,0.13143)
	(11.6,0.06476)
	(11.7,-0.0019)
	(11.8,-0.06857)
	(11.9,-0.13524)
	(12,-0.24731)
	(12.1,-0.60789)
	(12.2,-1.00789)
	(12.3,-1.40789)
	(12.4,-1.77642)
	(12.5,-1.91954)
	(12.6,-1.98677)
	(12.7,-2.01829)
	(12.8,-2.03611)
	(12.9,-2.0406)
	(13,-2.03175)
	(13.1,-2.00957)
	(13.2,-1.9714)
	(13.3,-1.88168)
	(13.4,-1.6777)
	(13.5,-1.2777)
	(13.6,-0.8777)
	(13.7,-0.4777)
	(13.8,-0.19289)
	(13.9,-0.11875)
	(14,-0.05208)
	\end{pspicture}
	\caption{Graphs of $p$ and $q$}\label{fig: 4}
\end{figure}

\begin{figure}[htb]
	\centering
	\psset{unit=1.25cm}
	\begin{pspicture}(-4,-3.75)(4.2,4)
	\psaxes[labels=none,linecolor=lightgray,ticks=none]{->}(0,0)(-4,-3.75)(4.2,4)
	\rput(1.,-0.4){$1$}
	\psline[linecolor=gray](1.0,0)(1.0,-0.1)
	\rput(-0.4,1){$1$}
	\psline[linecolor=gray](-1.0,0)(-1.0,-0.1)
	\rput(-1,-0.4){$-1$}
	\rput(-0.4,-1){$-1$}
	\rput(-0.4,-3){$-3$}
	\psline[linecolor=gray](-0.1,-3)(0,-3)
	\psline[linecolor=gray](-0.1,1)(0,1)
\psline[linecolor=gray](-0.1,-1)(0,-1)
	\psline[linecolor=gray](-0.1,3)(0,3)
	\psline[linecolor=gray](3,-0.1)(3,0)
	\rput(-3,-0.4){$-3$}
	\psline[linecolor=gray](-3,-0.1)(-3,0)
	\rput(3,-0.4){$3$}
	\rput(-0.4,3){$3$}
	\rput(3.9,-0.4){$x(t)$}
	\rput(-0.6,3.7){$x(t-1)$}
	\psline[linecolor=red](2.006,0)
	(2.00594,0.0066)
	(2.00574,0.01327)
	(2.00541,0.01993)
	(2.00494,0.0266)
	(2.00434,0.03327)
	(2.00361,0.03993)
	(2.00275,0.0466)
	(2.00175,0.05327)
	(2.00062,0.05993)
	(1.99935,0.0667)
	(1.99795,0.07426)
	(1.99638,0.0828)
	(1.99463,0.09245)
	(1.99267,0.10333)
	(1.99049,0.11559)
	(1.98804,0.12939)
	(1.9853,0.14492)
	(1.98224,0.16237)
	(1.9788,0.18199)
	(1.97494,0.20404)
	(1.97062,0.22881)
	(1.96578,0.25663)
	(1.96034,0.28787)
	(1.95424,0.32295)
	(1.9474,0.36203)
	(1.93977,0.40203)
	(1.93133,0.44203)
	(1.92209,0.48203)
	(1.91206,0.52203)
	(1.90122,0.56203)
	(1.88958,0.60203)
	(1.87715,0.64203)
	(1.86391,0.68203)
	(1.84987,0.72203)
	(1.83504,0.76203)
	(1.8194,0.80203)
	(1.80296,0.84203)
	(1.78573,0.88203)
	(1.76769,0.92203)
	(1.74815,0.96203)
	(1.7158,1.00203)
	(1.6758,1.04203)
	(1.6358,1.08203)
	(1.5958,1.12203)
	(1.5558,1.16203)
	(1.5158,1.20203)
	(1.4758,1.24203)
	(1.4358,1.28203)
	(1.3958,1.32203)
	(1.3558,1.36203)
	(1.3158,1.40203)
	(1.2758,1.44203)
	(1.2358,1.48203)
	(1.1958,1.52203)
	(1.1558,1.56203)
	(1.1158,1.60203)
	(1.0758,1.64203)
	(1.0358,1.68203)
	(0.9958,1.72195)
	(0.9558,1.75187)
	(0.9158,1.77077)
	(0.8758,1.78869)
	(0.8358,1.80581)
	(0.7958,1.82213)
	(0.7558,1.83765)
	(0.7158,1.85237)
	(0.6758,1.86629)
	(0.6358,1.87941)
	(0.5958,1.89173)
	(0.5558,1.90325)
	(0.5158,1.91397)
	(0.4758,1.92389)
	(0.4358,1.93301)
	(0.3958,1.94133)
	(0.3558,1.94885)
	(0.31736,1.95558)
	(0.2831,1.96158)
	(0.25262,1.96693)
	(0.22551,1.97171)
	(0.20142,1.97598)
	(0.18002,1.97979)
	(0.16103,1.9832)
	(0.14419,1.98625)
	(0.12928,1.98898)
	(0.11609,1.99144)
	(0.10445,1.99364)
	(0.09421,1.99562)
	(0.08523,1.99742)
	(0.07739,1.99904)
	(0.07052,2.00052)
	(0.06385,2.00187)
	(0.05718,2.00308)
	(0.05052,2.00415)
	(0.04385,2.0051)
	(0.03718,2.00591)
	(0.03052,2.00659)
	(0.02385,2.00713)
	(0.01718,2.00754)
	(0.01052,2.00782)
	(0.00385,2.00797)
	(-0.00282,2.00798)
	(-0.00948,2.00785)
	(-0.01615,2.0076)
	(-0.02282,2.00721)
	(-0.02948,2.00669)
	(-0.03615,2.00603)
	(-0.04282,2.00524)
	(-0.04948,2.00432)
	(-0.05615,2.00326)
	(-0.06282,2.00208)
	(-0.06948,2.00075)
	(-0.07626,1.9993)
	(-0.08391,1.9977)
	(-0.09269,1.99593)
	(-0.1027,1.99398)
	(-0.11408,1.99182)
	(-0.12698,1.98941)
	(-0.14157,1.98673)
	(-0.15805,1.98374)
	(-0.17664,1.9804)
	(-0.19758,1.97666)
	(-0.22116,1.97248)
	(-0.24768,1.9678)
	(-0.27751,1.96256)
	(-0.31104,1.95668)
	(-0.34872,1.9501)
	(-0.38872,1.94273)
	(-0.42872,1.93456)
	(-0.46872,1.92559)
	(-0.50872,1.91582)
	(-0.54872,1.90525)
	(-0.58872,1.89388)
	(-0.62872,1.8817)
	(-0.66872,1.86873)
	(-0.70872,1.85496)
	(-0.74872,1.84039)
	(-0.78872,1.82502)
	(-0.82872,1.80885)
	(-0.86872,1.79188)
	(-0.90872,1.77411)
	(-0.94872,1.75554)
	(-0.98872,1.72875)
	(-1.02872,1.68945)
	(-1.06872,1.64945)
	(-1.10872,1.60945)
	(-1.14872,1.56945)
	(-1.18872,1.52945)
	(-1.22872,1.48945)
	(-1.26872,1.44945)
	(-1.30872,1.40945)
	(-1.34872,1.36945)
	(-1.38872,1.32945)
	(-1.42872,1.28945)
	(-1.46872,1.24945)
	(-1.50872,1.20945)
	(-1.54872,1.16945)
	(-1.58872,1.12945)
	(-1.62872,1.08945)
	(-1.66872,1.04945)
	(-1.70872,1.00945)
	(-1.74388,0.96945)
	(-1.7648,0.92945)
	(-1.783,0.88945)
	(-1.80039,0.84945)
	(-1.81698,0.80945)
	(-1.83277,0.76945)
	(-1.84777,0.72945)
	(-1.86196,0.68945)
	(-1.87535,0.64945)
	(-1.88795,0.60945)
	(-1.89974,0.56945)
	(-1.91073,0.52945)
	(-1.92092,0.48945)
	(-1.93032,0.44945)
	(-1.93891,0.40945)
	(-1.9467,0.36945)
	(-1.9537,0.33011)
	(-1.95994,0.29468)
	(-1.96552,0.26318)
	(-1.9705,0.2352)
	(-1.97495,0.21036)
	(-1.97894,0.18834)
	(-1.98251,0.16883)
	(-1.98571,0.15158)
	(-1.98859,0.13635)
	(-1.99118,0.12295)
	(-1.99352,0.11119)
	(-1.99564,0.10091)
	(-1.99757,0.09198)
	(-1.99933,0.08428)
	(-2.00095,0.07752)
	(-2.00243,0.07086)
	(-2.00378,0.06419)
	(-2.005,0.05752)
	(-2.00608,0.05086)
	(-2.00703,0.04419)
	(-2.00785,0.03752)
	(-2.00854,0.03086)
	(-2.00909,0.02419)
	(-2.00951,0.01752)
	(-2.00979,0.01086)
	(-2.00994,0.00419)
	(-2.00996,-0.00248)
	(-2.00984,-0.00914)
	(-2.00959,-0.01581)
	(-2.00921,-0.02248)
	(-2.0087,-0.02914)
	(-2.00805,-0.03581)
	(-2.00727,-0.04248)
	(-2.00635,-0.04914)
	(-2.0053,-0.05581)
	(-2.00412,-0.06248)
	(-2.0028,-0.06914)
	(-2.00135,-0.07581)
	(-1.99977,-0.08249)
	(-1.99805,-0.08987)
	(-1.99617,-0.09844)
	(-1.99411,-0.10833)
	(-1.99183,-0.11967)
	(-1.98931,-0.1326)
	(-1.98652,-0.14729)
	(-1.98341,-0.16396)
	(-1.97995,-0.1828)
	(-1.97609,-0.20409)
	(-1.97177,-0.2281)
	(-1.96695,-0.25514)
	(-1.96155,-0.28559)
	(-1.9555,-0.31985)
	(-1.94873,-0.3583)
	(-1.94117,-0.3983)
	(-1.93281,-0.4383)
	(-1.92365,-0.4783)
	(-1.91369,-0.5183)
	(-1.90292,-0.5583)
	(-1.89136,-0.5983)
	(-1.879,-0.6383)
	(-1.86584,-0.6783)
	(-1.85188,-0.7183)
	(-1.83711,-0.7583)
	(-1.82155,-0.7983)
	(-1.80519,-0.8383)
	(-1.78803,-0.8783)
	(-1.77006,-0.9183)
	(-1.75097,-0.9583)
	(-1.72017,-0.9983)
	(-1.68019,-1.0383)
	(-1.64019,-1.0783)
	(-1.60019,-1.1183)
	(-1.56019,-1.1583)
	(-1.52019,-1.1983)
	(-1.48019,-1.2383)
	(-1.44019,-1.2783)
	(-1.40019,-1.3183)
	(-1.36019,-1.3583)
	(-1.32019,-1.3983)
	(-1.28019,-1.4383)
	(-1.24019,-1.4783)
	(-1.20019,-1.5183)
	(-1.16019,-1.5583)
	(-1.12019,-1.5983)
	(-1.08019,-1.6383)
	(-1.04019,-1.6783)
	(-1.00019,-1.7183)
	(-0.96019,-1.75007)
	(-0.92019,-1.76941)
	(-0.88019,-1.78742)
	(-0.84019,-1.80463)
	(-0.80019,-1.82104)
	(-0.76019,-1.83664)
	(-0.72019,-1.85145)
	(-0.68019,-1.86546)
	(-0.64019,-1.87867)
	(-0.60019,-1.89108)
	(-0.56019,-1.90268)
	(-0.52019,-1.91349)
	(-0.48019,-1.9235)
	(-0.44019,-1.93271)
	(-0.40019,-1.94111)
	(-0.36019,-1.94872)
	(-0.32171,-1.95554)
	(-0.28745,-1.96163)
	(-0.25703,-1.96707)
	(-0.23005,-1.97194)
	(-0.20614,-1.9763)
	(-0.18498,-1.9802)
	(-0.1663,-1.98372)
	(-0.14983,-1.98687)
	(-0.13537,-1.98973)
	(-0.12272,-1.9923)
	(-0.11171,-1.99465)
	(-0.10219,-1.99679)
	(-0.09404,-1.99875)
	(-0.08708,-2.00056)
	(-0.08042,-2.00223)
	(-0.07375,-2.00377)
	(-0.06708,-2.00518)
	(-0.06042,-2.00646)
	(-0.05375,-2.0076)
	(-0.04708,-2.00861)
	(-0.04042,-2.00949)
	(-0.03375,-2.01023)
	(-0.02708,-2.01084)
	(-0.02042,-2.01131)
	(-0.01375,-2.01166)
	(-0.00708,-2.01186)
	(-0.00042,-2.01194)
	(0.00625,-2.01188)
	(0.01292,-2.01169)
	(0.01958,-2.01137)
	(0.02625,-2.01091)
	(0.03292,-2.01032)
	(0.03958,-2.00959)
	(0.04625,-2.00874)
	(0.05292,-2.00775)
	(0.05958,-2.00662)
	(0.06625,-2.00536)
	(0.07292,-2.00397)
	(0.07958,-2.00245)
	(0.08625,-2.00079)
	(0.09309,-1.999)
	(0.10106,-1.99706)
	(0.11037,-1.99495)
	(0.12116,-1.99264)
	(0.13356,-1.9901)
	(0.14774,-1.98729)
	(0.16389,-1.98418)
	(0.18222,-1.98072)
	(0.20299,-1.97687)
	(0.22646,-1.97259)
	(0.25295,-1.9678)
	(0.28282,-1.96245)
	(0.31645,-1.95647)
	(0.35431,-1.94977)
	(0.39431,-1.94229)
	(0.43431,-1.93401)
	(0.47431,-1.92493)
	(0.51431,-1.91504)
	(0.55431,-1.90436)
	(0.59431,-1.89288)
	(0.63431,-1.8806)
	(0.67431,-1.86752)
	(0.71431,-1.85363)
	(0.75431,-1.83895)
	(0.79431,-1.82347)
	(0.83431,-1.80719)
	(0.87431,-1.79011)
	(0.91431,-1.77222)
	(0.95431,-1.75346)
	(0.99431,-1.72433)
	(1.03431,-1.68451)
	(1.07431,-1.64451)
	(1.11431,-1.60451)
	(1.15431,-1.56451)
	(1.19431,-1.52451)
	(1.23431,-1.48451)
	(1.27431,-1.44451)
	(1.31431,-1.40451)
	(1.35431,-1.36451)
	(1.39431,-1.32451)
	(1.43431,-1.28451)
	(1.47431,-1.24451)
	(1.51431,-1.20451)
	(1.55431,-1.16451)
	(1.59431,-1.12451)
	(1.63431,-1.08451)
	(1.67431,-1.04451)
	(1.71431,-1.00451)
	(1.74777,-0.96451)
	(1.76775,-0.92451)
	(1.78584,-0.88451)
	(1.80314,-0.84451)
	(1.81963,-0.80451)
	(1.83532,-0.76451)
	(1.85022,-0.72451)
	(1.86431,-0.68451)
	(1.87761,-0.64451)
	(1.8901,-0.60451)
	(1.9018,-0.56451)
	(1.91269,-0.52451)
	(1.92278,-0.48451)
	(1.93208,-0.44451)
	(1.94057,-0.40451)
	(1.94827,-0.36451)
	(1.95517,-0.32582)
	(1.96134,-0.29134)
	(1.96685,-0.26075)
	(1.9718,-0.23365)
	(1.97623,-0.20967)
	(1.98021,-0.18849)
	(1.98379,-0.16983)
	(1.98702,-0.15344)
	(1.98994,-0.1391)
	(1.9926,-0.12661)
	(1.99502,-0.11582)
	(1.99724,-0.10658)
	(1.99929,-0.09876)
	(2.0012,-0.092)
	(2.00297,-0.08534)
	(2.00462,-0.07867)
	(2.006,0)
	\psline[linecolor=blue](3.0252,-0.9702)
	(3.04629,-0.9302)
	(3.0645,-0.8902)
	(3.0819,-0.8502)
	(3.09851,-0.8102)
	(3.11432,-0.7702)
	(3.12933,-0.7302)
	(3.14354,-0.6902)
	(3.15694,-0.6502)
	(3.16955,-0.6102)
	(3.18136,-0.5702)
	(3.19237,-0.5302)
	(3.20258,-0.4902)
	(3.21198,-0.4502)
	(3.22059,-0.4102)
	(3.2284,-0.3702)
	(3.23541,-0.3302)
	(3.24162,-0.2902)
	(3.24702,-0.2502)
	(3.25163,-0.2102)
	(3.25544,-0.1702)
	(3.25845,-0.1302)
	(3.26066,-0.0902)
	(3.26206,-0.0502)
	(3.26267,-0.0102)
	(3.26248,0.0298)
	(3.26149,0.0698)
	(3.2597,0.1098)
	(3.2571,0.1498)
	(3.25371,0.1898)
	(3.24952,0.2298)
	(3.24453,0.2698)
	(3.23874,0.3098)
	(3.23214,0.3498)
	(3.22475,0.3898)
	(3.21656,0.4298)
	(3.20757,0.4698)
	(3.19778,0.5098)
	(3.18718,0.5498)
	(3.17579,0.5898)
	(3.1636,0.6298)
	(3.15061,0.6698)
	(3.13682,0.7098)
	(3.12222,0.7498)
	(3.10683,0.7898)
	(3.09064,0.8298)
	(3.07365,0.8698)
	(3.05586,0.9098)
	(3.03726,0.9498)
	(3.01003,0.9898)
	(2.9706,1.0298)
	(2.9306,1.0698)
	(2.8906,1.1098)
	(2.8506,1.1498)
	(2.8106,1.1898)
	(2.7706,1.2298)
	(2.7306,1.2698)
	(2.6906,1.3098)
	(2.6506,1.3498)
	(2.6106,1.3898)
	(2.5706,1.4298)
	(2.5306,1.4698)
	(2.4906,1.5098)
	(2.4506,1.5498)
	(2.4106,1.5898)
	(2.3706,1.6298)
	(2.3306,1.6698)
	(2.2906,1.7098)
	(2.2506,1.7498)
	(2.2106,1.7898)
	(2.1706,1.8298)
	(2.1306,1.8698)
	(2.0906,1.9098)
	(2.0506,1.9498)
	(2.02367,1.9898)
	(2.0161,2.0298)
	(2.00943,2.0698)
	(2.00277,2.1098)
	(1.9961,2.1498)
	(1.98943,2.1898)
	(1.98277,2.2298)
	(1.9761,2.2698)
	(1.96943,2.3098)
	(1.96277,2.3498)
	(1.9561,2.3898)
	(1.94943,2.4298)
	(1.94277,2.4698)
	(1.9361,2.5098)
	(1.92943,2.5498)
	(1.92277,2.5898)
	(1.9161,2.6298)
	(1.90943,2.6698)
	(1.90277,2.7098)
	(1.8961,2.7498)
	(1.88943,2.7898)
	(1.88277,2.8298)
	(1.8761,2.8698)
	(1.86943,2.9098)
	(1.86277,2.9498)
	(1.8561,2.9898)
	(1.84943,3.0252)
	(1.84277,3.04629)
	(1.8361,3.0645)
	(1.82943,3.0819)
	(1.82277,3.09851)
	(1.8161,3.11432)
	(1.80943,3.12933)
	(1.80277,3.14354)
	(1.7961,3.15694)
	(1.78943,3.16955)
	(1.78277,3.18136)
	(1.7761,3.19237)
	(1.76943,3.20258)
	(1.76277,3.21198)
	(1.7561,3.22059)
	(1.74943,3.2284)
	(1.74277,3.23541)
	(1.7361,3.24162)
	(1.72943,3.24702)
	(1.72277,3.25163)
	(1.7161,3.25544)
	(1.70943,3.25845)
	(1.70277,3.26066)
	(1.6961,3.26206)
	(1.68943,3.26267)
	(1.68277,3.26248)
	(1.6761,3.26149)
	(1.66943,3.2597)
	(1.66277,3.2571)
	(1.6561,3.25371)
	(1.64943,3.24952)
	(1.64277,3.24453)
	(1.6361,3.23874)
	(1.62943,3.23214)
	(1.62277,3.22475)
	(1.6161,3.21656)
	(1.60943,3.20757)
	(1.60277,3.19778)
	(1.5961,3.18718)
	(1.58943,3.17579)
	(1.58277,3.1636)
	(1.5761,3.15061)
	(1.56943,3.13682)
	(1.56277,3.12222)
	(1.5561,3.10683)
	(1.54943,3.09064)
	(1.54277,3.07365)
	(1.5361,3.05586)
	(1.52943,3.03726)
	(1.52277,3.01003)
	(1.5161,2.9706)
	(1.50943,2.9306)
	(1.50277,2.8906)
	(1.4961,2.8506)
	(1.48943,2.8106)
	(1.48277,2.7706)
	(1.4761,2.7306)
	(1.46943,2.6906)
	(1.46277,2.6506)
	(1.4561,2.6106)
	(1.44943,2.5706)
	(1.44277,2.5306)
	(1.4361,2.4906)
	(1.42943,2.4506)
	(1.42277,2.4106)
	(1.4161,2.3706)
	(1.40943,2.3306)
	(1.40277,2.2906)
	(1.3961,2.2506)
	(1.38943,2.2106)
	(1.38277,2.1706)
	(1.3761,2.1306)
	(1.36943,2.0906)
	(1.36277,2.0506)
	(1.3561,2.02367)
	(1.34943,2.0161)
	(1.34277,2.00943)
	(1.3361,2.00277)
	(1.32869,1.9961)
	(1.31722,1.98943)
	(1.30131,1.98277)
	(1.28095,1.9761)
	(1.25615,1.96943)
	(1.22691,1.96277)
	(1.19322,1.9561)
	(1.1551,1.94943)
	(1.1151,1.94277)
	(1.0751,1.9361)
	(1.0351,1.92943)
	(0.9951,1.92277)
	(0.9551,1.9161)
	(0.9151,1.90943)
	(0.8751,1.90277)
	(0.8351,1.8961)
	(0.7951,1.88943)
	(0.7551,1.88277)
	(0.7151,1.8761)
	(0.6751,1.86943)
	(0.6351,1.86277)
	(0.5951,1.8561)
	(0.5551,1.84943)
	(0.5151,1.84277)
	(0.4751,1.8361)
	(0.4351,1.82943)
	(0.3951,1.82277)
	(0.3551,1.8161)
	(0.3151,1.80943)
	(0.2751,1.80277)
	(0.2351,1.7961)
	(0.1951,1.78943)
	(0.1551,1.78277)
	(0.1151,1.7761)
	(0.0751,1.76943)
	(0.0351,1.76277)
	(-0.0049,1.7561)
	(-0.0449,1.74943)
	(-0.0849,1.74277)
	(-0.1249,1.7361)
	(-0.1649,1.72943)
	(-0.2049,1.72277)
	(-0.2449,1.7161)
	(-0.2849,1.70943)
	(-0.3249,1.70277)
	(-0.3649,1.6961)
	(-0.4049,1.68943)
	(-0.4449,1.68277)
	(-0.4849,1.6761)
	(-0.5249,1.66943)
	(-0.5649,1.66277)
	(-0.6049,1.6561)
	(-0.6449,1.64943)
	(-0.6849,1.64277)
	(-0.7249,1.6361)
	(-0.7649,1.62943)
	(-0.8049,1.62277)
	(-0.8449,1.6161)
	(-0.8849,1.60943)
	(-0.9249,1.60277)
	(-0.9649,1.5961)
	(-1.0049,1.58943)
	(-1.0449,1.58277)
	(-1.0849,1.5761)
	(-1.1249,1.56943)
	(-1.1649,1.56277)
	(-1.2049,1.5561)
	(-1.2449,1.54943)
	(-1.2849,1.54277)
	(-1.3249,1.5361)
	(-1.3649,1.52943)
	(-1.4049,1.52277)
	(-1.4449,1.5161)
	(-1.4849,1.50943)
	(-1.5249,1.50277)
	(-1.5649,1.4961)
	(-1.6049,1.48943)
	(-1.6449,1.48277)
	(-1.6849,1.4761)
	(-1.7249,1.46943)
	(-1.7649,1.46277)
	(-1.8049,1.4561)
	(-1.8449,1.44943)
	(-1.8849,1.44277)
	(-1.9249,1.4361)
	(-1.9649,1.42943)
	(-2.0049,1.42277)
	(-2.0449,1.4161)
	(-2.0849,1.40943)
	(-2.1249,1.40277)
	(-2.1649,1.3961)
	(-2.2049,1.38943)
	(-2.2449,1.38277)
	(-2.2849,1.3761)
	(-2.3249,1.36943)
	(-2.3649,1.36277)
	(-2.4049,1.3561)
	(-2.4449,1.34943)
	(-2.4849,1.34277)
	(-2.5249,1.3361)
	(-2.5649,1.32869)
	(-2.6049,1.31722)
	(-2.6449,1.30131)
	(-2.6849,1.28095)
	(-2.7249,1.25615)
	(-2.7649,1.22691)
	(-2.8049,1.19322)
	(-2.8449,1.1551)
	(-2.8849,1.1151)
	(-2.9249,1.0751)
	(-2.9649,1.0351)
	(-3.00478,0.9951)
	(-3.03441,0.9551)
	(-3.05326,0.9151)
	(-3.07116,0.8751)
	(-3.08827,0.8351)
	(-3.10457,0.7951)
	(-3.12008,0.7551)
	(-3.13479,0.7151)
	(-3.14869,0.6751)
	(-3.1618,0.6351)
	(-3.1741,0.5951)
	(-3.18561,0.5551)
	(-3.19632,0.5151)
	(-3.20622,0.4751)
	(-3.21533,0.4351)
	(-3.22363,0.3951)
	(-3.23114,0.3551)
	(-3.23785,0.3151)
	(-3.24375,0.2751)
	(-3.24886,0.2351)
	(-3.25316,0.1951)
	(-3.25667,0.1551)
	(-3.25938,0.1151)
	(-3.26128,0.0751)
	(-3.26239,0.0351)
	(-3.26269,-0.0049)
	(-3.2622,-0.0449)
	(-3.26091,-0.0849)
	(-3.25881,-0.1249)
	(-3.25592,-0.1649)
	(-3.25222,-0.2049)
	(-3.24773,-0.2449)
	(-3.24244,-0.2849)
	(-3.23634,-0.3249)
	(-3.22945,-0.3649)
	(-3.22175,-0.4049)
	(-3.21326,-0.4449)
	(-3.20397,-0.4849)
	(-3.19387,-0.5249)
	(-3.18298,-0.5649)
	(-3.17128,-0.6049)
	(-3.15879,-0.6449)
	(-3.1455,-0.6849)
	(-3.1314,-0.7249)
	(-3.11651,-0.7649)
	(-3.10081,-0.8049)
	(-3.08432,-0.8449)
	(-3.06703,-0.8849)
	(-3.04893,-0.9249)
	(-3.02896,-0.9649)
	(-2.9955,-1.0049)
	(-2.9555,-1.0449)
	(-2.9155,-1.0849)
	(-2.8755,-1.1249)
	(-2.8355,-1.1649)
	(-2.7955,-1.2049)
	(-2.7555,-1.2449)
	(-2.7155,-1.2849)
	(-2.6755,-1.3249)
	(-2.6355,-1.3649)
	(-2.5955,-1.4049)
	(-2.5555,-1.4449)
	(-2.5155,-1.4849)
	(-2.4755,-1.5249)
	(-2.4355,-1.5649)
	(-2.3955,-1.6049)
	(-2.3555,-1.6449)
	(-2.3155,-1.6849)
	(-2.2755,-1.7249)
	(-2.2355,-1.7649)
	(-2.1955,-1.8049)
	(-2.1555,-1.8449)
	(-2.1155,-1.8849)
	(-2.0755,-1.9249)
	(-2.0373,-1.9649)
	(-2.02025,-2.0049)
	(-2.01358,-2.0449)
	(-2.00692,-2.0849)
	(-2.00025,-2.1249)
	(-1.99358,-2.1649)
	(-1.98692,-2.2049)
	(-1.98025,-2.2449)
	(-1.97358,-2.2849)
	(-1.96692,-2.3249)
	(-1.96025,-2.3649)
	(-1.95358,-2.4049)
	(-1.94692,-2.4449)
	(-1.94025,-2.4849)
	(-1.93358,-2.5249)
	(-1.92692,-2.5649)
	(-1.92025,-2.6049)
	(-1.91358,-2.6449)
	(-1.90692,-2.6849)
	(-1.90025,-2.7249)
	(-1.89358,-2.7649)
	(-1.88692,-2.8049)
	(-1.88025,-2.8449)
	(-1.87358,-2.8849)
	(-1.86692,-2.9249)
	(-1.86025,-2.9649)
	(-1.85358,-3.00478)
	(-1.84692,-3.03441)
	(-1.84025,-3.05326)
	(-1.83358,-3.07116)
	(-1.82692,-3.08827)
	(-1.82025,-3.10457)
	(-1.81358,-3.12008)
	(-1.80692,-3.13479)
	(-1.80025,-3.14869)
	(-1.79358,-3.1618)
	(-1.78692,-3.1741)
	(-1.78025,-3.18561)
	(-1.77358,-3.19632)
	(-1.76692,-3.20622)
	(-1.76025,-3.21533)
	(-1.75358,-3.22363)
	(-1.74692,-3.23114)
	(-1.74025,-3.23785)
	(-1.73358,-3.24375)
	(-1.72692,-3.24886)
	(-1.72025,-3.25316)
	(-1.71358,-3.25667)
	(-1.70692,-3.25938)
	(-1.70025,-3.26128)
	(-1.69358,-3.26239)
	(-1.68692,-3.26269)
	(-1.68025,-3.2622)
	(-1.67358,-3.26091)
	(-1.66692,-3.25881)
	(-1.66025,-3.25592)
	(-1.65358,-3.25222)
	(-1.64692,-3.24773)
	(-1.64025,-3.24244)
	(-1.63358,-3.23634)
	(-1.62692,-3.22945)
	(-1.62025,-3.22175)
	(-1.61358,-3.21326)
	(-1.60692,-3.20397)
	(-1.60025,-3.19387)
	(-1.59358,-3.18298)
	(-1.58692,-3.17128)
	(-1.58025,-3.15879)
	(-1.57358,-3.1455)
	(-1.56692,-3.1314)
	(-1.56025,-3.11651)
	(-1.55358,-3.10081)
	(-1.54692,-3.08432)
	(-1.54025,-3.06703)
	(-1.53358,-3.04893)
	(-1.52692,-3.02896)
	(-1.52025,-2.9955)
	(-1.51358,-2.9555)
	(-1.50692,-2.9155)
	(-1.50025,-2.8755)
	(-1.49358,-2.8355)
	(-1.48692,-2.7955)
	(-1.48025,-2.7555)
	(-1.47358,-2.7155)
	(-1.46692,-2.6755)
	(-1.46025,-2.6355)
	(-1.45358,-2.5955)
	(-1.44692,-2.5555)
	(-1.44025,-2.5155)
	(-1.43358,-2.4755)
	(-1.42692,-2.4355)
	(-1.42025,-2.3955)
	(-1.41358,-2.3555)
	(-1.40692,-2.3155)
	(-1.40025,-2.2755)
	(-1.39358,-2.2355)
	(-1.38692,-2.1955)
	(-1.38025,-2.1555)
	(-1.37358,-2.1155)
	(-1.36692,-2.0755)
	(-1.36025,-2.0373)
	(-1.35358,-2.02025)
	(-1.34692,-2.01358)
	(-1.34025,-2.00692)
	(-1.33358,-2.00025)
	(-1.32488,-1.99358)
	(-1.31173,-1.98692)
	(-1.29415,-1.98025)
	(-1.27211,-1.97358)
	(-1.24563,-1.96692)
	(-1.21471,-1.96025)
	(-1.17935,-1.95358)
	(-1.14,-1.94692)
	(-1.1,-1.94025)
	(-1.06,-1.93358)
	(-1.02,-1.92692)
	(-0.98,-1.92025)
	(-0.94,-1.91358)
	(-0.9,-1.90692)
	(-0.86,-1.90025)
	(-0.82,-1.89358)
	(-0.78,-1.88692)
	(-0.74,-1.88025)
	(-0.7,-1.87358)
	(-0.66,-1.86692)
	(-0.62,-1.86025)
	(-0.58,-1.85358)
	(-0.54,-1.84692)
	(-0.5,-1.84025)
	(-0.46,-1.83358)
	(-0.42,-1.82692)
	(-0.38,-1.82025)
	(-0.34,-1.81358)
	(-0.3,-1.80692)
	(-0.26,-1.80025)
	(-0.22,-1.79358)
	(-0.18,-1.78692)
	(-0.14,-1.78025)
	(-0.1,-1.77358)
	(-0.06,-1.76692)
	(-0.02,-1.76025)
	(0.02,-1.75358)
	(0.06,-1.74692)
	(0.1,-1.74025)
	(0.14,-1.73358)
	(0.18,-1.72692)
	(0.22,-1.72025)
	(0.26,-1.71358)
	(0.3,-1.70692)
	(0.34,-1.70025)
	(0.38,-1.69358)
	(0.42,-1.68692)
	(0.46,-1.68025)
	(0.5,-1.67358)
	(0.54,-1.66692)
	(0.58,-1.66025)
	(0.62,-1.65358)
	(0.66,-1.64692)
	(0.7,-1.64025)
	(0.74,-1.63358)
	(0.78,-1.62692)
	(0.82,-1.62025)
	(0.86,-1.61358)
	(0.9,-1.60692)
	(0.94,-1.60025)
	(0.98,-1.59358)
	(1.02,-1.58692)
	(1.06,-1.58025)
	(1.1,-1.57358)
	(1.14,-1.56692)
	(1.18,-1.56025)
	(1.22,-1.55358)
	(1.26,-1.54692)
	(1.3,-1.54025)
	(1.34,-1.53358)
	(1.38,-1.52692)
	(1.42,-1.52025)
	(1.46,-1.51358)
	(1.5,-1.50692)
	(1.54,-1.50025)
	(1.58,-1.49358)
	(1.62,-1.48692)
	(1.66,-1.48025)
	(1.7,-1.47358)
	(1.74,-1.46692)
	(1.78,-1.46025)
	(1.82,-1.45358)
	(1.86,-1.44692)
	(1.9,-1.44025)
	(1.94,-1.43358)
	(1.98,-1.42692)
	(2.02,-1.42025)
	(2.06,-1.41358)
	(2.1,-1.40692)
	(2.14,-1.40025)
	(2.18,-1.39358)
	(2.22,-1.38692)
	(2.26,-1.38025)
	(2.3,-1.37358)
	(2.34,-1.36692)
	(2.38,-1.36025)
	(2.42,-1.35358)
	(2.46,-1.34692)
	(2.5,-1.34025)
	(2.54,-1.33358)
	(2.58,-1.32488)
	(2.62,-1.31173)
	(2.66,-1.29415)
	(2.7,-1.27211)
	(2.74,-1.24563)
	(2.78,-1.21471)
	(2.82,-1.17935)
	(2.86,-1.14)
	(2.9,-1.1)
	(2.94,-1.06)
	(2.98,-1.02)
	(3.01794,-0.98)
	(3.04171,-0.94)
	(3.06011,-0.9)
	(3.07771,-0.86)
	(3.09452,-0.82)
	(3.11052,-0.78)
	(3.12573,-0.74)
	(3.14013,-0.7)
	(3.15373,-0.66)
	(3.16654,-0.62)
	(3.17854,-0.58)
	(3.18975,-0.54)
	(3.20015,-0.5)
	(3.20975,-0.46)
	(3.21856,-0.42)
	(3.22656,-0.38)
	(3.23377,-0.34)
	(3.24017,-0.3)
	(3.24577,-0.26)
	(3.25058,-0.22)
	(3.25458,-0.18)
	(3.25779,-0.14)
	(3.26019,-0.1)
	(3.26179,-0.06)
	(3.2626,-0.02)
	(3.2626,0.02)
	(3.26181,0.06)
	(3.26021,0.1)
	(3.25781,0.14)
	(3.25462,0.18)
	(3.25062,0.22)
	(3.24583,0.26)
	(3.24023,0.3)
	(3.23383,0.34)
	(3.22664,0.38)
	(3.21864,0.42)
	(3.20985,0.46)
	(3.20025,0.5)
	(3.18985,0.54)
	(3.17866,0.58)
	(3.16666,0.62)
	(3.15387,0.66)
	(3.14027,0.7)
	(3.12587,0.74)
	(3.11068,0.78)
	(3.09468,0.82)
	(3.07789,0.86)
	(3.06029,0.9)
	\end{pspicture}
	\caption{Traces of $p$ and $q$ in the $(x(t),x(t-1))$ plane}\label{fig: 5}
\end{figure}
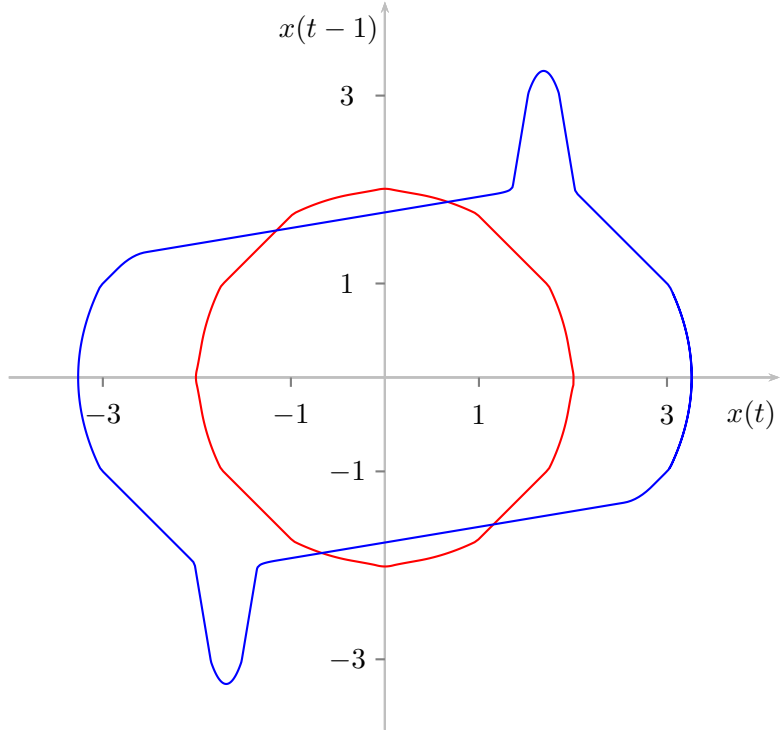

\section{Third example: many SOP solutions} \label{SEC5}

An idea similar to that in this section is described at the end of Vas's paper \cite{Vas 2011}.

Observe that, other than the differentiability of $f$ at 0 and the negative feedback condition itself, the only restriction on $f$ in $[0,2a]$ required in Proposition \ref{prop: example 2} is that $f(x) \in [-\gamma,0)$ for all $x \in (0,2a]$.   This observation now allows us to specify $f$ at different ``scales'' and so create an equation with as many periodic solutions as we wish.

Let $\gamma_n$, $n \in \N$, be a sequence of positive reals satisfying $\gamma_{n+1} < \gamma_n/4$.  Then, for each $\gamma_n$, let us take (for example)
\[
a_n = \gamma_n/4, \ \delta_n = \gamma_n/8, \ c_n = \gamma_n/64.
\]
Then $\gamma_n,\ a_n,\ \delta_n,\ c_n$ together satisfy conditions (i) -- (iii) above.  Given any $M \in \N \cup \lbrace\infty\rbrace$, with $M > 1$, there exists a bounded, odd, continuous function $f$ satisfying the negative feedback condition and satisfying the following version of (\ref{eq: assumptions 3}) for all $n < M$:
\begin{equation}
  \left\{\begin{array}{l}\mbox{$f$ is odd and $|f| \leq \gamma_n$ on $[0,2a_n]$;}\\[0.2cm]
\mbox{$f(x)=\begin{cases}-\gamma_n,& \text{for }x\in[a_n,2a_n-c_n];\\
-\delta_n,& \text{for }x\in[2a_n,\gamma_n].\end{cases}$}\\
  \end{array} \right.
\end{equation}
If $M < \infty$, we can moreover take $f$ to be differentiable at $0$.  According to our work last section, for such $f$ Eq. \eqref{eq: proto} has $M-1$ stable SOP solutions whose segments are fixed points of $P$.

Suppose we take $\gamma_1 = 5$ and $\gamma_2 = 1$, and $a_n, c_n, \delta_n$ as described above for $n \in \{1,2\}$.  We take $f$ piecewise linear.  The two stable periodic solutions are readily seen in numerical approximation; see figure below.

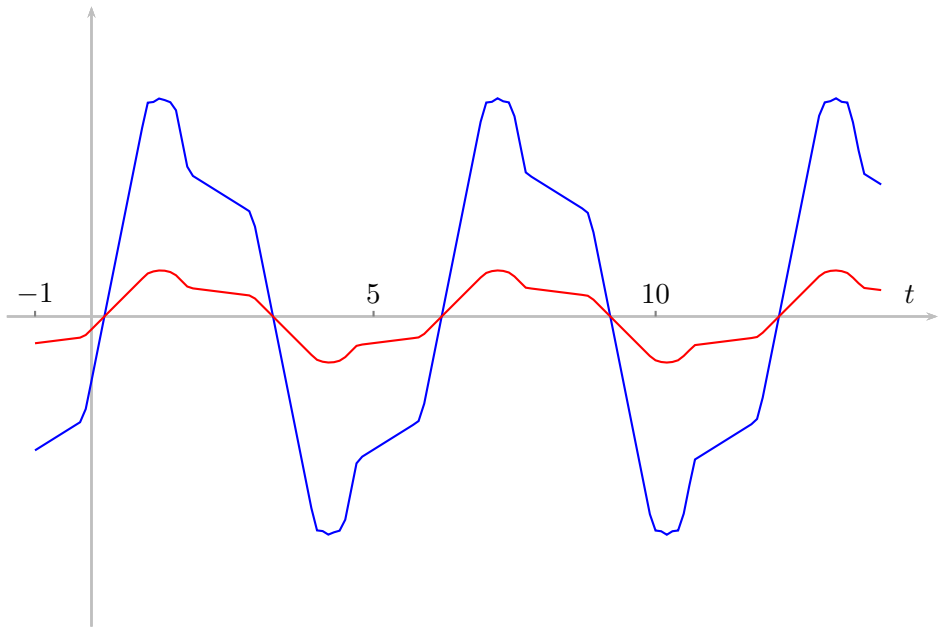
\begin{figure}[htb]
	\centering
	\psset{unit=0.75cm}
	\begin{pspicture}(-2,-6)(15,6)
	\psaxes[labels=none,linecolor=lightgray,ticks=none]{->}(0,0)(-1.5,-5.5)(15.,5.5)
	\rput(-1.,0.4){$-1$}
	\psline[linecolor=gray](-1.0,0)(-1.0,0.1)
	\rput(5,0.4){$5$}
	\psline[linecolor=gray](5.0,0)(5.0,0.1)
	\rput(10,0.4){$10$}
	\psline[linecolor=gray](10,0)(10,0.1)
	\rput(14.5,0.4){$t$}
	
	\psline[linecolor=blue](-1.001,-2.37468)
	(-0.902,-2.31281)
	(-0.802,-2.25031)
	(-0.702,-2.18781)
	(-0.602,-2.12531)
	(-0.502,-2.06281)
	(-0.402,-2.00031)
	(-0.302,-1.93781)
	(-0.202,-1.87531)
	(-0.102,-1.64126)
	(-0.002,-1.15308)
	(0.098,-0.65308)
	(0.198,-0.15308)
	(0.298,0.34692)
	(0.398,0.84692)
	(0.498,1.34692)
	(0.598,1.84692)
	(0.698,2.34692)
	(0.798,2.84692)
	(0.898,3.34692)
	(0.998,3.79294)
	(1.098,3.80544)
	(1.198,3.8664)
	(1.298,3.83908)
	(1.398,3.80072)
	(1.498,3.65811)
	(1.598,3.15811)
	(1.698,2.65811)
	(1.798,2.49363)
	(1.898,2.43113)
	(1.998,2.36863)
	(2.098,2.30613)
	(2.198,2.24363)
	(2.298,2.18113)
	(2.398,2.11863)
	(2.498,2.05613)
	(2.598,1.99363)
	(2.698,1.93113)
	(2.798,1.86698)
	(2.898,1.59554)
	(2.998,1.09964)
	(3.098,0.59964)
	(3.198,0.09964)
	(3.298,-0.40036)
	(3.398,-0.90036)
	(3.498,-1.40036)
	(3.598,-1.90036)
	(3.698,-2.40036)
	(3.798,-2.90036)
	(3.898,-3.40036)
	(3.998,-3.79431)
	(4.098,-3.80681)
	(4.198,-3.86919)
	(4.298,-3.82842)
	(4.398,-3.79942)
	(4.498,-3.60465)
	(4.598,-3.10465)
	(4.698,-2.60465)
	(4.798,-2.48698)
	(4.898,-2.42448)
	(4.998,-2.36198)
	(5.098,-2.29948)
	(5.198,-2.23698)
	(5.298,-2.17448)
	(5.398,-2.11198)
	(5.498,-2.04948)
	(5.598,-1.98698)
	(5.698,-1.92448)
	(5.798,-1.85475)
	(5.898,-1.54608)
	(5.998,-1.04646)
	(6.098,-0.54646)
	(6.198,-0.04646)
	(6.298,0.45354)
	(6.398,0.95354)
	(6.498,1.45354)
	(6.598,1.95354)
	(6.698,2.45354)
	(6.798,2.95354)
	(6.898,3.45354)
	(6.998,3.79559)
	(7.098,3.80809)
	(7.198,3.87075)
	(7.298,3.81774)
	(7.398,3.79803)
	(7.498,3.55143)
	(7.598,3.05143)
	(7.698,2.55619)
	(7.798,2.48027)
	(7.898,2.41777)
	(7.998,2.35527)
	(8.098,2.29277)
	(8.198,2.23027)
	(8.298,2.16777)
	(8.398,2.10527)
	(8.498,2.04277)
	(8.598,1.98027)
	(8.698,1.91777)
	(8.798,1.83839)
	(8.898,1.49281)
	(8.998,0.99281)
	(9.098,0.49281)
	(9.198,-0.00719)
	(9.298,-0.50719)
	(9.398,-1.00719)
	(9.498,-1.50719)
	(9.598,-2.00719)
	(9.698,-2.50719)
	(9.798,-3.00719)
	(9.898,-3.50719)
	(9.998,-3.79694)
	(10.098,-3.80956)
	(10.198,-3.87123)
	(10.298,-3.80921)
	(10.398,-3.79671)
	(10.498,-3.49782)
	(10.598,-2.99782)
	(10.698,-2.53608)
	(10.798,-2.47358)
	(10.898,-2.41108)
	(10.998,-2.34858)
	(11.098,-2.28608)
	(11.198,-2.22358)
	(11.298,-2.16108)
	(11.398,-2.09858)
	(11.498,-2.03608)
	(11.598,-1.97358)
	(11.698,-1.91108)
	(11.798,-1.81806)
	(11.898,-1.43929)
	(11.998,-0.93929)
	(12.098,-0.43929)
	(12.198,0.06071)
	(12.298,0.56071)
	(12.398,1.06071)
	(12.498,1.56071)
	(12.598,2.06071)
	(12.698,2.56071)
	(12.798,3.06071)
	(12.898,3.56071)
	(12.998,3.79831)
	(13.098,3.81965)
	(13.198,3.8706)
	(13.298,3.80791)
	(13.398,3.79541)
	(13.498,3.44431)
	(13.598,2.94431)
	(13.698,2.52943)
	(13.798,2.46693)
	(13.898,2.40443)
	(13.998,2.34193)
	\psline[linecolor=red](-1.001,-0.47494)
	(-0.902,-0.46257)
	(-0.802,-0.45007)
	(-0.702,-0.43757)
	(-0.602,-0.42507)
	(-0.502,-0.41257)
	(-0.402,-0.40007)
	(-0.302,-0.38757)
	(-0.202,-0.37507)
	(-0.102,-0.32828)
	(-0.002,-0.23065)
	(0.098,-0.13065)
	(0.198,-0.03065)
	(0.298,0.06935)
	(0.398,0.16935)
	(0.498,0.26935)
	(0.598,0.36935)
	(0.698,0.46935)
	(0.798,0.56935)
	(0.898,0.66935)
	(0.998,0.76348)
	(1.098,0.79971)
	(1.198,0.81594)
	(1.298,0.81216)
	(1.398,0.78839)
	(1.498,0.73165)
	(1.598,0.63165)
	(1.698,0.53165)
	(1.798,0.49873)
	(1.898,0.48623)
	(1.998,0.47373)
	(2.098,0.46123)
	(2.198,0.44873)
	(2.298,0.43623)
	(2.398,0.42373)
	(2.498,0.41123)
	(2.598,0.39873)
	(2.698,0.38623)
	(2.798,0.3734)
	(2.898,0.31911)
	(2.998,0.21993)
	(3.098,0.11993)
	(3.198,0.01993)
	(3.298,-0.08007)
	(3.398,-0.18007)
	(3.498,-0.28007)
	(3.598,-0.38007)
	(3.698,-0.48007)
	(3.798,-0.58007)
	(3.898,-0.68007)
	(3.998,-0.76832)
	(4.098,-0.80241)
	(4.198,-0.81649)
	(4.298,-0.81058)
	(4.398,-0.78467)
	(4.498,-0.72093)
	(4.598,-0.62093)
	(4.698,-0.52093)
	(4.798,-0.49739)
	(4.898,-0.48489)
	(4.998,-0.47239)
	(5.098,-0.45989)
	(5.198,-0.44739)
	(5.298,-0.43489)
	(5.398,-0.42239)
	(5.498,-0.40989)
	(5.598,-0.39739)
	(5.698,-0.38489)
	(5.798,-0.37095)
	(5.898,-0.30921)
	(5.998,-0.20928)
	(6.098,-0.10928)
	(6.198,-0.00928)
	(6.298,0.09072)
	(6.398,0.19072)
	(6.498,0.29072)
	(6.598,0.39072)
	(6.698,0.49072)
	(6.798,0.59072)
	(6.898,0.69072)
	(6.998,0.77289)
	(7.098,0.80485)
	(7.198,0.81681)
	(7.298,0.80877)
	(7.398,0.78072)
	(7.498,0.71028)
	(7.598,0.61028)
	(7.698,0.51124)
	(7.798,0.49606)
	(7.898,0.48356)
	(7.998,0.47106)
	(8.098,0.45856)
	(8.198,0.44606)
	(8.298,0.43356)
	(8.398,0.42106)
	(8.498,0.40856)
	(8.598,0.39606)
	(8.698,0.38356)
	(8.798,0.36768)
	(8.898,0.29858)
	(8.998,0.19858)
	(9.098,0.09858)
	(9.198,-0.00142)
	(9.298,-0.10142)
	(9.398,-0.20142)
	(9.498,-0.30142)
	(9.598,-0.40142)
	(9.698,-0.50142)
	(9.798,-0.60142)
	(9.898,-0.70142)
	(9.998,-0.77727)
	(10.098,-0.80709)
	(10.198,-0.8169)
	(10.298,-0.80672)
	(10.398,-0.77654)
	(10.498,-0.69959)
	(10.598,-0.59959)
	(10.698,-0.50722)
	(10.798,-0.49472)
	(10.898,-0.48222)
	(10.998,-0.46972)
	(11.098,-0.45722)
	(11.198,-0.44472)
	(11.298,-0.43222)
	(11.398,-0.41972)
	(11.498,-0.40722)
	(11.598,-0.39472)
	(11.698,-0.38222)
	(11.798,-0.36361)
	(11.898,-0.28787)
	(11.998,-0.18787)
	(12.098,-0.08787)
	(12.198,0.01213)
	(12.298,0.11213)
	(12.398,0.21213)
	(12.498,0.31213)
	(12.598,0.41213)
	(12.698,0.51213)
	(12.798,0.61213)
	(12.898,0.71213)
	(12.998,0.78143)
	(13.098,0.8091)
	(13.198,0.81677)
	(13.298,0.80445)
	(13.398,0.77212)
	(13.498,0.68887)
	(13.598,0.58887)
	(13.698,0.50588)
	(13.798,0.49338)
	(13.898,0.48088)
	(13.998,0.46838)
	\end{pspicture}
	\caption{Two stable periodic solutions}\label{fig: 6}
\end{figure}

\end{document}